\documentclass[10 pt, leqno]{article}

\date{}
\usepackage{amssymb, amsbsy, amsmath, amsfonts, amssymb, amscd, mathrsfs, amsthm}

\usepackage[english]{babel}
\usepackage[latin1]{inputenc}
\usepackage[T1]{fontenc}
\usepackage{indentfirst}
\usepackage{color}

\usepackage{ifpdf}
\ifpdf
	\usepackage[pdftex]{graphicx}
\else
	\usepackage[dvips]{graphicx} 
\fi

\usepackage{caption}

\makeatletter
\@addtoreset{equation}{section}
\makeatother

\newtheorem{statement}{}[section]
\newtheorem{theorem}[statement]{Theorem}
\newtheorem{lemma}[statement]{Lemma}

\newtheorem{proposition}[statement]{Proposition}
\newtheorem{definition}[statement]{Definition}
\newtheorem{corollary}[statement]{Corollary}

\newcommand\C{\mathbb C}

\newcommand\R{\mathbb R}
\newcommand\T{\mathbb T}
\newcommand\D{\mathbb D}

\newcommand\e{{\rm e}}

\newcommand\eps{\varepsilon}
\newcommand\ind{1 \kern - 0.28 em {\rm I}}
\newcommand\dis{\displaystyle}
\renewcommand \Re{{\mathfrak R}{\rm e}\,}
\renewcommand \Im{{\mathfrak I}{\rm m}\,}
\newcommand\converge{\mathop{\longrightarrow}\limits}

\let\phi=\varphi

\let\tilde=\widetilde
\newcommand\tq{\, ; \ }

\title{\bf Compactification and decompactification by weights on Bergman spaces}
\author{\it Pascal~Lef\`evre, Daniel~Li,  \\ \it Herv\'e~Queff\'elec, Luis~Rodr{\'\i}guez-Piazza}

\date{\footnotesize \today}

\begin{document}

\maketitle

\noindent {\bf Abstract.} We characterize the symbols $\phi$ for which there exists a weight $w$ such that the weighted composition operator 
$M_w C_\phi$ is compact on the weighted Bergman space ${\mathfrak B}^2_\alpha$. We also characterize the symbols for which there exists a 
weight $w$ such that $M_w C_\phi$ is bounded but not compact. We also investigate when there exists $w$ such that $M_w C_\phi$ is Hilbert-Schmidt on 
${\mathfrak B}^2_\alpha$. 
\medskip

\smallskip


\section {Introduction} 

It is known (see \cite{GKP} for instance) that ``weightening'' a composition operator $C_\phi$ on the Hardy space $H^2$ by some weight $w$, we can improve 
its compactness properties, and even its membership in Schatten classes $S_p$, or the decay of its approximation numbers (\cite[Theorem~2.3]{GDHL}, 
\cite{LLQR-compactification}), or at the opposite make a compact composition operator non compact (\cite{LLQR-compactification}). 

In this paper, we consider weighted composition operators $M_w C_\phi$ on the weighted Bergman spaces ${\mathfrak B}^2_\alpha$, with $\alpha > - 1$. 
Note that for such an operator to be bounded from ${\mathfrak B}^2_\alpha$ into itself, it is necessary that $w \in {\mathfrak B}^2_\alpha$ (since 
$w = (M_w C_\phi) (\ind)$). 

We show in Section~\ref{section compactif} that $C_\phi$ can be weighted to become compact on ${\mathfrak B}^2_\alpha$ if and only if the set where 
$\phi$ has an angular derivative has null measure. 

In Section~\ref{section decompact}, we show that the exists a weight $w$ such that $M_w C_\phi$ is bounded but not compact on 
${\mathfrak B}^2_\alpha$ if and only if $\| \phi \|_\infty = 1$.

In Section~\ref{section HS}, we study when $M_w C_\phi$ can be Hilbert-Schmidt on ${\mathfrak B}^2_\alpha$ for some weight $w$.

\section {Notation and background}

The weighted Bergman space ${\mathfrak B}^2_\alpha$, with $\alpha > - 1$, is the space of all analytic functions $f \colon \D \to \C$ on the unit disk $\D$ 
such that
\begin{displaymath} 
\| f \|_{{\mathfrak B}_\alpha^2}^2 = (\alpha + 1) \int_\D |f (z)|^2 (1 - |z|^2)^\alpha \, dA (z) < \infty \, ,
\end{displaymath} 
where $A$ is the normalized area measure on $\D$. When $\alpha = 0$, we write simply ${\mathfrak B}^2$ instead of ${\mathfrak B}^2_0$ and call it 
the Bergman space.  
\medskip

Every analytic self-map $\phi \colon \D \to \D$ defines a bounded composition operator $C_\phi \colon f \mapsto f \circ \phi$ from 
${\mathfrak B}^2_\alpha$ into itself (\cite[Proposition~3.4]{MacCluer-Shapiro}). 

The pull-back measure $A_\phi$ of $\phi$ is defined as:
\begin{displaymath} 
\qquad \qquad A_\phi (B) = A [\phi^{- 1} (B)] \quad \text{for all Borel sets } B \subseteq \D \, .
\end{displaymath} 

Let $\mu$ be a finite Borel measure on $\D$. For $\beta > 1$, the measure $\mu$ is said a \emph{$\beta$-Carleson measure} if:
\begin{equation} 
\sup_{|\xi| = 1} \mu [S (\xi, h)] = {\rm O}\, (h^\beta) \, ,
\end{equation} 
where
\begin{displaymath} 
S (\xi, h) = \{ z \in \D \tq |z - \xi| < h\} 
\end{displaymath} 
is the Carleson box of size $h$ centered at $\xi \in \T = \partial \D$. The measure $\mu$ is said a \emph{vanishing $\beta$-Carleson measure} if:
\begin{equation} 
\qquad \quad \sup_{|\xi| = 1} \mu [S (\xi, h)] = {\rm o}\, (h^\beta) \quad \text{as } {h \to 0} \, .
\end{equation} 

Recall the following result (see \cite{Hastings} and \cite[Theorem~4.3]{MacCluer-Shapiro}).
\begin{theorem} 
Let $\mu$ be a finite Borel measure on $\D$. Then:
\begin{itemize}
\setlength\itemsep {-0.1 em}

\item [{\rm (a)}] ${\mathfrak B}^2_\alpha \subseteq L^2 (\mu)$ if and only if $\mu$ is an $(\alpha + 2)$-Carleson measure. 

Moreover, when this happens, the canonical inclusion $J_\mu \colon {\mathfrak B}^2_\alpha \to L^2 (\mu)$ is bounded.

\item [{\rm (b)}] The canonical inclusion $J_\mu \colon {\mathfrak B}^2_\alpha \to L^2 (\mu)$ is compact if and only if $\mu$ is a vanishing 
$(\alpha + 2)$-Carleson measure. 
\end{itemize}
\end{theorem} 
\begin{corollary} \label{coro Carleson} 
Let $\phi \colon \D \to \D$ be an analytic self-map and $w \in {\mathfrak B}^2_\alpha$. Set, for every Borel set $B$ in $\D$: 
\begin{equation} \label{mesure mu}
\mu_{w, \phi} (B) = \int_{\phi^{- 1} (B)} |w (z)|^2 (1 - |z|^2)^\alpha \, dA (z) \, .
\end{equation} 
Then:
\begin{enumerate} 
\setlength\itemsep {-0.1 em}

\item [{\rm (a)}] The weighted composition operator $M_w C_\phi \colon {\mathfrak B}^2_\alpha \to {\mathfrak B}^2_\alpha$, defined as:
\begin{equation} 
(M_w C_\phi) f = w \, (f \circ \phi) \, , 
\end{equation} 
is bounded if and only if $\mu_{w, \phi}$ is an $(\alpha + 2)$-Carleson measure.

\item [{\rm (b)}] The weighted composition operator $M_w C_\phi \colon {\mathfrak B}^2_\alpha \to {\mathfrak B}^2_\alpha$ is compact 
if and only if $\mu_{w, \phi}$ is a vanishing $(\alpha + 2)$-Carleson measure.
\end{enumerate}
\end{corollary} 
\begin{proof}
Observe that, for all $f \in {\mathfrak B}^2_\alpha$, we have
\begin{displaymath} 
\| (M_w C_\phi) f \|_{{\mathfrak B}^2_\alpha}^2 = \int_\D | f [\phi (z) ] |^2 \, |w (z) |^2 (1 - |z|^2)^\alpha \, dA (z) 
= \| f \|_{L^2 (\mu_{w, \phi})}^2 \, . \qedhere
\end{displaymath} 
\end{proof}
%

\section {Compactification} \label{section compactif} 

Recall the following definitions (see \cite[Section~4.1]{Shapiro-livre}).

\begin{definition} 
A holomorphic self-map $\phi \colon \D \to \D$ has an \emph{angular limit} {\rm (}or a \emph{non-tangential limit}{\rm )} $l$ at $\xi \in \T$ if $\phi (z)$ 
converges to $l$ whenever $z$ tends to $\xi$ inside any angular sector in $\D$ whose vertex is $\xi$. Then $l$ is called the \emph{angular limit} of $\phi$ at $\xi$ 
and is denoted:
\begin{displaymath} 
l = \angle \lim_{z \to \xi} \phi (z) \, .
\end{displaymath} 
\end{definition}
\begin{definition} \label{angular derivative}
A holomorphic self-map $\phi \colon \D \to \D$ has an \emph{angular derivative} at $\xi \in \T$ if it has an angular limit $\zeta$ at $\xi$, 
\emph{with $|\zeta| = 1$} and:
\begin{displaymath} 
\angle \lim_{z \to \xi} \frac{\phi (z) - \zeta}{z - \xi}  
\end{displaymath} 
exists and is finite. This limit is called the angular derivative of $\phi$ at $\xi$ and is denoted by $\phi ' (\xi)$. 
\end{definition}

Let us also recall that the Julia-Carath\'eodory theorem (see \cite[Section~4.2]{Shapiro-livre}), says that $\phi$ has an angular derivative at $\xi \in \T$ if and 
only if:
\begin{equation} 
\delta := \liminf_{z \to \xi} \frac{1 - |\phi (z)|}{1 - |z|} < + \infty \, ,
\end{equation} 
or, equivalently:
\begin{equation} \label{J-C}
\limsup_{z \to \xi} \frac{1 - |z|}{1 - |\phi (z)|} > 0 \, ,
\end{equation} 
and, when this happens, we have $\delta > 0$ and $\phi ' (\xi) = \xi \, \bar{\zeta} \, \delta$, so $|\phi ' (\xi) | = \delta$. 
\medskip

We define
\begin{equation} \label{AD}
{\mathcal {AD}} (\phi) = \{ \xi \in \T \tq \text{$\phi$ has an angular derivative at $\xi$} \} \, ,
\end{equation} 
and we call it the \emph{angular derivative set of $\phi$}. 
\medskip

B.~MacCluer and J.~Shapiro proved (\cite[Theorem~3.5]{MacCluer-Shapiro}) that, for $\alpha > - 1$, the composition operator 
$C_\phi \colon {\mathfrak B}^2_\alpha \to {\mathfrak B}^2_\alpha$ is compact if and only if:
\begin{equation} \label{compact}
{\mathcal {AD}} (\phi) = \emptyset \, . 
\end{equation} 

Asking for a compactification, we have the following result. 
\medskip
\goodbreak

\begin{theorem}\label{compactif}
Let $\phi \colon \D \to \D$ be an analytic self-map. Then the following assertions are equivalent, for the weighted Bergman space ${\mathfrak B}^2_\alpha$, 
with $\alpha > - 1$:
\begin{enumerate} 
\setlength\itemsep {-0.1 em}

\item [$1)$] there exists a holomorphic function $w$, with $w \not\equiv 0$, such that the weighted composition operator 
$M_w C_\phi \colon {\mathfrak B}^2_\alpha \to {\mathfrak B}^2_\alpha$ is compact; 

\item [$2)$] there exists a weight $w \in H^\infty$, with $w \not\equiv 0$, such that the weighted composition operator 
$M_w C_\phi \colon {\mathfrak B}^2_\alpha \to {\mathfrak B}^2_\alpha$ is compact;

\item [$3)$] the angular derivative set of $\phi$ has null measure: 
\noindent
\begin{equation} \label{compactifiable}
m [{\mathcal {AD}} (\phi)] = 0 \, , 
\end{equation} 
where $m$ is the normalized Lebesgue measure on $\T = \partial \D$; 

\item [$4)$] $\dis \lim_{z \to \xi} \frac{1 - |z|}{1 - |\phi (z)|} = 0$ for almost all $\xi \in \T$.
\end{enumerate}
\end{theorem} 
\goodbreak

For example, if $\phi (z) = \frac{1 + z}{2}\,$, then $C_\phi$ is not compact on ${\mathfrak B}^2_\alpha$, but it is compactifiable by a weight in $H^\infty$.
\smallskip

When the equivalent conditions of Theorem~\ref{compactif} are satisfied, we say that composition operator $C_\phi$ is \emph{compactifiable}.
\smallskip

The proof will be based on the following result of Moorhouse (\cite[Corollary~1]{Moorhouse}; see also  \cite[Proposition~1]{Cuc-Zhao}).

\begin{proposition} [Moorhouse] \label{prop Moorhouse}
Let $\alpha > - 1$. Let $\phi$ and $w$ be analytic functions on $\D$. Then:
\begin{itemize}
\setlength\itemsep {-0.1 em}

\item [$1)$] If the weighted composition operator $M_w C_\phi$ is compact on ${\mathfrak B}^2_\alpha$, we have:
\begin{equation} \label{cond necess Moorhouse}
\lim_{|z| \to 1} |w (z) |^2 \bigg(\frac{1 - |z|^2}{1 - | \phi (z)|^2}\bigg)^{\alpha + 2} = 0 \, .
\end{equation} 

\item [$2)$] When $w$ is bounded, $M_w C_\phi$ is compact on ${\mathfrak B}^2_\alpha$ if and only if 
\begin{equation} \label{cond Moorhouse}
\lim_{|z| \to 1} |w (z) |^2 \frac{1 - |z|^2}{1 - | \phi (z)|^2} = 0 \, .
\end{equation} 
\end{itemize}
\end{proposition}

For $1)$, we compute: 
\begin{displaymath} 
\| (M_w C_\phi)^\ast (k_z) \|_{( {\mathfrak B}^2_\alpha)^\ast}^2 
= |w (z)|^2  \bigg( \frac{1 - |z|^2}{1 - |\phi (z)|^2} \bigg)^{\alpha + 2} \, \raise 1pt \hbox{,}
\end{displaymath} 
where $k_z$ is the normalized reproducing kernel of ${\mathfrak B}^2_\alpha$, and, using that $k_z$ weakly converges to $0$ as $|z| \to 1$, we obtain:
\begin{displaymath} 
\lim_{|z| \to 1} |w (z)|^2 \bigg( \frac{1 - |z|^2}{1 - |\phi (z)|^2} \bigg)^{\alpha + 2} = 0 \, .
\end{displaymath} 
To obtain the necessary condition in 2), we use the following, easily checked, fact which shows that \eqref{cond Moorhouse} is equivalent to 
\eqref{cond necess Moorhouse}, since $w$ is bounded in $2)$. 
\goodbreak
\begin{lemma}
Let $f, g \colon \D \to [0, \infty)$ be two bounded functions. Then the following assertions are equivalent:
\begin{itemize}
\setlength\itemsep {-0.1 em}

\item [{\rm a)}] $\lim_{|z| \to 1} f (z) \, g (z) = 0$;

\item [{\rm b)}]  $\lim_{|z| \to 1} \min [ f (z), g (z)] = 0$;

\item [{\rm c)}]  $\lim_{|z| \to 1} [f (z)]^a [g (z)]^b = 0$, for all $a, b > 0$. 
\end{itemize}
\end{lemma}

The sufficient condition in $2)$ is proved by \cite[Lemma~1]{Moorhouse}.

\begin{proof} [Proof of Theorem~\ref{compactif}] 
The implication $2) \Rightarrow 1)$ needs no comment.
\smallskip

\noindent $3) \Rightarrow 2)$ Assume that $m [{\mathcal{AD}} (\phi)] = 0$. A theorem of Privalov (see \cite[Vol.~I, bottom of page 276]{Zyg}),
 asserts the existence a function $w \not\equiv 0$ in $H^\infty$ such that
\begin{equation} 
\qquad \lim_{z \to \xi} w (z) = 0 \quad \text{for all } \xi \in {\mathcal{AD}} (\phi) \, .
\end{equation} 

The Schwarz-Pick lemma (see \cite[Corollary~2.40]{Co-MC}) tells that 
\begin{displaymath} 
\frac{1 - |z|^2}{1 - |\phi (z)|^2} \leq 2 \, \frac{1 - |z|}{1 - |\phi (z)|} \leq 2 \, \frac{1 + |\phi (0)|}{1 - |\phi (0)|}  \, \cdot
\end{displaymath} 
Hence for $\xi \in {\mathcal{AD}} (\phi)$, we have
\begin{displaymath} 
\lim_{z \to \xi} |w (z)|^2 \, \frac{1 - |z|^2}{1 - |\phi (z)|^2} = 0 \, . 
\end{displaymath} 
For $\xi \notin {\mathcal{AD}} (\phi)$, thanks to the Julia-Carath\'eodory theorem and \eqref{J-C}, we also have
\begin{displaymath} 
\lim_{z \to \xi} |w (z)|^2 \, \frac{1 - |z|^2}{1 - |\phi (z)|^2} = 0 \, .
\end{displaymath} 
Hence
\begin{equation} \label{for all}
\qquad\qquad \lim_{z \to \xi} |w (z)|^2 \, \frac{1 - |z|^2}{1 - |\phi (z)|^2} = 0 \quad \text{for all } \xi \in \T \, .
\end{equation} 
By a compactness argument, we obtain that
\begin{equation} \label{compactness argument} 
\lim_{|z| \to 1} |w (z)|^2 \, \frac{1 - |z|^2}{1 - |\phi (z)|^2} = 0 \, ;
\end{equation} 
in fact, if \eqref{compactness argument} failed, there would be a sequence $(z_n)$ such that $|z_n| \converge_{n \to \infty} 1$ and for which
\begin{displaymath} 
\limsup_{n \to \infty}  |w (z_n) |^2 \, \frac{1 - |z_n|^2}{1 - |\phi (z_n)|^2} > 0 \, ; 
\end{displaymath} 
by compactness a subsequence converges to some $\xi \in \partial\D$, and that would contradict \eqref{for all}. 

Since $w$ is bounded, it follows from Proposition~\ref{prop Moorhouse}, that $M_w C_\phi$ is compact on ${\mathfrak B}^2_\alpha$. 
\smallskip

\noindent $1) \Rightarrow 3)$ Assume that $M_w C_\phi$ is compact with $w$ analytic and $w \not\equiv 0$.  By 
Proposition~\ref{prop Moorhouse},\,$1)$, we have:
\begin{equation} \label{Moorhouse} 
\lim_{|z| \to 1} |w (z)|^2 \bigg(\frac{1 - |z|^2}{1 - |\phi (z)|^2}\bigg)^{\alpha + 2} = 0 \, .
\end{equation} 
In particular, for every $\xi \in \T$:
\begin{equation} \label{Moorhouse-bis}
\lim_{z \to \xi} |w (z)|^2 \bigg( \frac{1 - |z|^2}{1 - |\phi (z)|^2} \bigg)^{\alpha  + 2}= 0 \, .
\end{equation} 
Now, for every $\xi \in {\mathcal{AD}} (\phi)$, we have, if $\zeta$ is the angular limit of $\phi$ at $\xi$:
\begin{displaymath} 
\lim_{z \to \xi} \bigg| \frac{\phi (z) - \zeta}{z - \xi} \bigg| = |\phi ' (\xi) | < \infty \, .
\end{displaymath} 
If $z$ belongs to an angular sector $S_\xi$ of vertex $\xi$, there is a positive constant $C$, depending only on this sector, such that $|z - \xi| \leq C (1 - |z|)$; 
hence
\begin{displaymath} 
\liminf_{z \to \xi, z \in S_\xi} \frac{1 - |z|}{1 - |\phi (z) |} \geq \liminf_{z \to \xi, z \in S_\xi} C \, \frac{|z - \xi |}{|\phi (z) - \zeta|} 
= \frac{C}{|\phi ' (\xi)|} > 0 \, .
\end{displaymath} 
Then, it follows, with \eqref{Moorhouse-bis}, that $\lim_{z \to \xi, z \in S_\xi} w (z) = 0$. Since the angular sector $S_\xi$ is arbitrary, we get that
$\angle \lim_{z \to \xi} w (z) = 0$. 
\smallskip

By another theorem of Privalov (see \cite[Chapter~XIV, Theorem~(1.1) and Theorem~(1.9)]{Zyg}, 
\cite[Chapter~VI, Theorem~2.3]{Garnett-Marshall}, or \cite[Chapter~II, Exercise 10]{Garnett}, where it is called ``local Fatou theorem'') , it follows, since 
$w \not\equiv 0$, that $m [{\mathcal{AD}} (\phi)] = 0$. 
\smallskip

\noindent $3) \Longleftrightarrow 4)$ follows from the Julia-Caratheodory theorem, as stated in \eqref{J-C}.
\end{proof}

\noindent {\bf Remark~1.} The implication $2) \Rightarrow 3)$ can be proved using the classical F. and M. Riesz theorem (see \cite[Theorem~2.2]{Duren}) 
instead of Privalov's theorem.
\smallskip

\noindent {\bf Remark~2.} Condition \eqref{Moorhouse} is necessary for the compactness of $M_w C_\phi$; however, it is not sufficient in general 
without this assumption that $w \in H^\infty$. An example is given in \cite[Section~5, Corollary~4]{Cuc-Zhao} for which $M_w C_\phi$, with $w = \phi '$, 
is not even bounded on ${\mathfrak B}^2$.

\section {Decompactification} \label{section decompact}

\subsection {The main result} \label{subsection main} 

In the sequel, as usual, $\alpha > - 1$.

\begin{definition}
We say that the composition operator $C_\phi \colon {\mathfrak B}^2_\alpha \to {\mathfrak B}^2_\alpha$ is \emph{decompactifiable} if there exists a weight 
$w \in {\mathfrak B}^2$ such that the weighted composition operator $M_w C_\phi \colon {\mathfrak B}^2_\alpha \to {\mathfrak B}^2_\alpha$ is bounded 
but not compact.
\end{definition}

Our main result is the following.
\goodbreak

\begin{theorem} \label{theo decompact} 
Let $\phi \colon \D \to \D$ be an analytic self-map. Then the composition operator $C_\phi \colon {\mathfrak B}^2_\alpha \to {\mathfrak B}^2_\alpha$ is 
decompactifiable if and only if $\| \phi \|_\infty = 1$. 
\end{theorem} 

It is a consequence of this other theorem, whose proof is postponed.

\begin{theorem} \label{theo clef}
Let $\gamma > 1$ and $\nu$ be a vanishing $\gamma$-Carleson measure on $\D$. Assume that $\nu$ satisfies the following property:
\begin{equation} \label{condition}
\forall t > 0\, , \quad \exists \zeta \in \partial \D \quad \text{such that } \nu [S (\zeta, t)] > 0 \,. 
\end{equation} 

Then there exists a holomorphic function $u \colon \D \to \C$ such that 
\begin{itemize}
\setlength\itemsep{-0.1 em}

\item [$(i)$] $u \in L^2 (\nu)$;

\item [$(ii)$] $\displaystyle \sup_{|\xi| = 1, \, 0 < h \leq 1}\frac{1}{h^\gamma} \int_{S (\xi, h)} |u|^2 \, d \nu < \infty$;

\item [$(iii)$] there exist $\delta > 0$ and two sequences $(\zeta_n)$ in $\partial\D$ and $(t_n)$ in $(0, 1)$ with $t_n \converge_{n \to \infty} 0^+$ such that
\begin{equation} 
\qquad \qquad \frac{1}{t_n^\gamma} \int_{S (\zeta_n, t_n)} |u|^2 \, d \nu\geq \delta \, , \quad \text{for all } n \geq 1 \, .
\end{equation} 
\end{itemize}
\end{theorem} 
\begin{proof} [Proof of Theorem~\ref{theo decompact}]
It is plain that if $\| \phi \|_\infty < 1$, then $M_w C_\phi$ is compact for every weight $w \in {\mathfrak B}^2_\alpha $. In fact, if $\mu_{w, \phi}$ is the 
measure defined in \eqref{mesure mu}, then $\mu_{w, \phi} [S (\xi, h)] = 0$ for $0 < h < 1 - \| \phi \|_\infty$; hence Corollary~\ref{coro Carleson} gives 
the result. 
\smallskip

Conversely, assume that $\| \phi \|_\infty = 1$. 

Note that if $C_\phi \colon {\mathfrak B}^2_\alpha  \to {\mathfrak B}^2_\alpha $ is not compact, it suffices to take $w = \ind$; so we assume that 
$C_\phi$ is compact. Then $\nu := (A_\alpha )_\phi = \phi (dA_\alpha)$ is a vanishing $(\alpha + 2)$-Carleson measure. 

Since $\| \phi\|_\infty = 1$, condition~\eqref{condition} is satisfied. Set $\gamma = \alpha + 2$ and  $u$ be the holomorphic function given by 
Theorem~\ref{theo clef} and set $w = u \circ \phi$. We have
\begin{displaymath} 
\int_\D |w|^2 \, dA_\alpha = \int_\D |u \circ \phi|^2 \, dA_\alpha = \int_\D |u|^2 \, d\nu < \infty \, ;
\end{displaymath} 
so $w \in {\mathfrak B}^2_\alpha$. 

Now, for every $\xi \in \partial \D$ and $h \in [0, 1)$, we have, with $\mu = \varphi (|w|^2 dA_\alpha)$:
\begin{align*} 
\mu [S (\xi, h)] 
& = \int_{\phi^{- 1} [S (\xi, h)]} |w|^2 \, dA_\alpha = \int_\D (\ind_{S (\xi, h)} \circ \phi) \, |u \circ \phi|^2 \, dA_\alpha \\
& = \int_\D \ind_{S (\xi, h)} |u|^2 \, d\nu \, .
\end{align*} 
Hence the properties $(ii)$ and $(iii)$ of Theorem~\ref{theo clef} show that $\mu$ is a non-vanishing $(\alpha + 2)$-Carleson measure, and therefore that 
$M_w C_\phi \colon {\mathfrak B}^2_\alpha \to {\mathfrak B}^2_\alpha$ is bounded but not compact. 
\end{proof}
%

\subsection {Proof of Theorem~\ref{theo clef}} \label{subsection proof} 

To prove Theorem~\ref{theo clef}, we need several auxiliary results. 

\begin{lemma} \label{lemma 1} 
For every $\omega \in \partial \D$ and $r \in (0, 1)$, there exists a bounded  analytic function $F \in H^\infty$ such that, for all $z \in \D$:
\begin{itemize}
\setlength\itemsep{-0.1 em}

\item [{\rm a)}] $\Re F (z) > 0$;

\item [{\rm b)}] $1 / 2 \leq | F (z) | \leq 2$; 

\item [{\rm c)}] $|F (z) | < 1$ when $| z - \omega| > r$; 

\item [{\rm d)}] $| F (z)| > 1$ when $| z - \omega | < r$.
\end{itemize}
\end{lemma}
\begin{proof}
By composing with a rotation, we can, and do, assume that $\omega = 1$. 

Let $C = 1 - r$ and let $A$ and $B$ be the points of the intersection of the unit circle $\T = \partial \D$ with the circle of center $1$ and radius $r$, with 
$\Im A > 0$ and $\Im B < 0$. Consider the M\"obius transformation $T$ sending $A$ to $0$, $C$ to $1$, and $B$ to $\infty$. The images by $T$ of 
$\partial \D$ and $\partial D (1, r)$ are straight lines passing through $0$. In fact the image of $\partial D (1, r)$ is the extended real line 
$\R_\infty = \R \cup \{\infty\}$. Moreover $T [ D (1, r)]$ is the open upper half-plane. 

Define $ g (z) = \sqrt{\,T (z)}$, where $\sqrt{\phantom{z}}$ is the principal branch of the square root. 

Then, for $z \in \D$:
\begin{displaymath}
\left\{
\begin{array}{ll}
\arg \, [g (z)] \in (0, \pi / 2)  & \quad \text{if } z \in D (1 , r) \, ,\smallskip\\
\arg \, [g (z)] \in (- \pi / 2, 0)  & \quad \text{if } z \in \D \setminus \overline{D (1, r)} \, .
\end{array}
\right.
\end{displaymath}

Let now $U$ be the M\"obius transformation sending $0$ to $i/2$, $\infty$ to $- i /2$, and $1$ to $0$. We have
\begin{itemize}
\setlength\itemsep{-0.1 em}

\item [{\bf --}] $| U [g (z)]| < 1/2$ \quad for all $z \in \D$;

\item [{\bf --}] $\Re U [g (z)] > 0$ \quad for all $z \in \D \cap D (1, r) = S (1, r)$;

\item [{\bf --}] $\Re U [g (z)] < 0$ \quad for all $z \in \D \setminus \overline{D (1, r)}$.
\end{itemize}

Finally, the function $F$ defined as $F (z) = \exp U [ g (z)]$ suits.
\end{proof}
\begin{lemma} \label{lemma 2}
Let $\gamma\ge1$,  $\nu$ be a $\gamma$-Carleson measure on $\D$ and $F \in H^\infty$ such that $1/2 \leq | F (z)| \leq 2$ for all $z \in \D$. For given 
$\beta \in (0, 1]$, we define the function $\Phi \colon \R_+^\ast \to \R_+$ as:
\begin{displaymath}
\Phi (\delta) = \sup_{|\xi| = 1; \, 0 < h \leq \beta} \frac{1}{h^\gamma} \int_{S (\xi, h)} |F|^{2 \delta} \, d\nu \, , 
\quad \text{for all } \delta > 0 \, .
\end{displaymath}
Then $\Phi$ is continuous. 
\end{lemma}
\goodbreak

\begin{proof}
First, we have $\Phi (\delta) < + \infty$ for all $\delta > 0$ because $\nu$ is a $\gamma$-Carlseon measure; indeed, for all $\xi \in \partial \D$ 
and all $h \in (0, 1]$:
\begin{displaymath}
\frac{1}{h^\gamma} \int_{S (\xi, h)} |F|^{2 \delta} \, d\nu \leq 4^\delta \, \frac{\nu [S (\xi, h)]}{h^\gamma} \leq C\, 4^\delta < + \infty \, .
\end{displaymath}

Now, observe that, since $1/2 \leq |F (z)| \leq 2$, we have, for all $h \in (0, 1]$, all $\xi \in \partial \D$, and all $t \in \R$:
\begin{displaymath}
\frac{1}{4^{|t|}} \, \frac{1}{h^\gamma} \int_{S (\xi, h)} |F|^{2 \delta} \, d\nu 
\leq \frac{1}{h^\gamma} \int_{S (\xi, h)} |F|^{2 \delta + 2 t} \, d\nu 
\leq 4^{|t|} \, \frac{1}{h^\gamma} \int_{S (\xi, h)} |F|^{2 \delta} \, d\nu \, .
\end{displaymath}
Taking the supremum, we get
\begin{displaymath}
4^{- |t|} \, \Phi (\delta) \leq \Phi (\delta + t) \leq 4^{|t|} \, \Phi (\delta) \, ,
\end{displaymath}
and that proves the continuity of $\Phi$, since $\Phi (\delta) < + \infty$.
\end{proof}
\begin{proposition} \label{prop 3} 
Let $\nu$ be a finite $\gamma$-Carleson measure on $\D$ with property \eqref{condition}. 
Then, for every $\beta \in (0, 1]$ and every $\eps \in (0, 1)$, there exists a function $v \in H^\infty$ satisfying:
\begin{itemize}
\setlength \itemsep{-0.1 em} 

\item [{\rm (a)}] $| v (z)| < \eps$ for all $z \in \D$ such that  $|z| < 1 - \beta$;

\item [{\rm (b)}] $\displaystyle \frac{1}{h^\gamma} \int_{S (\xi, h)} |v|^2 \, d\nu \leq 1$ for all $h \in (0, 1]$ and all $\xi \in \partial\D$;

\item [{\rm (c)}] $\displaystyle \frac{1}{h^\gamma} \int_{S (\xi, h)} |v|^2 \, d\nu \leq \eps^2$ for all $h \in (\beta, 1]$ and all $\xi \in \partial\D$;

\item [{\rm (d)}] there exists $t \in (0, \beta]$ and $\zeta \in \partial \D$ such that
\begin{displaymath}
\frac{1}{t^\gamma} \int_{S (\zeta, t)} |v|^2 \, d\nu \geq \Big( \frac{3}{4} \Big)^2 \, . \qquad \phantom{bla}
\end{displaymath}
\end{itemize}
\end{proposition}
\begin{proof}
Since $\nu$ is a $\gamma$-Carleson measure, there exists a positive constant $C$ (and we can and do assume that $C \geq 1$) such that:
\begin{equation} \label{2-Carleson}
\qquad\qquad\qquad  \nu [ S (\xi, h)] \leq C \, h^\gamma \, , \qquad \forall h \in (0, 1] \, , \ \forall \xi \in \partial\D \, .
\end{equation}

Take $r = \beta \,(\eps^2 /2)^{1 / \gamma}$. 
\smallskip

By \eqref{condition}, there exists $\omega \in \partial \D$ such that
\begin{displaymath}
\nu [S (\omega, r)] > 0 \, .
\end{displaymath}

Let $F$ be the function given by Lemma~\ref{lemma 1}. 

We define:
\begin{displaymath}
\Phi (\delta) = \sup_{0 < h \leq \beta, \, |\xi | = 1} \frac{1}{h^\gamma} \int_{S (\xi, h)} |F|^{2 \delta} \, d\nu \, .
\end{displaymath}
Thanks to \eqref{2-Carleson}, we have, for all $\xi \in \partial \D$ and all $h \in (0, \beta]$:
\begin{displaymath}
\frac{1}{h^\gamma} \int_{S (\xi, h)} |F|^{2} \, d\nu \leq 4 \, \frac{\nu [S (\xi, h)]}{h^\gamma} \leq 4 \, C \, ,
\end{displaymath}
and we get $\Phi (1) \leq 4 \, C$.

On the other hand, for all $\delta > 0$:
\begin{displaymath}
\Phi (\delta) \geq \frac{1}{r^\gamma} \int_{S (\omega, r)} |F|^{2 \delta} \, d \nu \, .
\end{displaymath}
Since $|F (z)| > 1$ for $z \in S (\omega, r)$ and $\nu [S (\omega, r)] > 0$, we get
\begin{displaymath}
\lim_{\delta \to + \infty} \int_{S (\omega, r)} |F|^{2 \delta} \, d \nu = + \infty \, ,
\end{displaymath}
and consequently $\lim_{\delta \to + \infty} \Phi (\delta) = + \infty$. Since $\Phi (1) \leq 4 \, C < (2 C/ \eps)^2$ and, thanks to 
Lemma~\ref{lemma 2}, $\Phi$ is continuous, there exists $\delta_0 > 1$ such that $\Phi (\delta_0) = (2 C/ \eps)^2$. 

Define
\begin{displaymath}
v = (\eps/2 C) \, F^{\delta_0} \, .
\end{displaymath}

Observe that $r < \beta$; so $|z| < 1 - \beta$ implies $|z| <  1 - r$; hence $z \notin \overline{S (\omega, r)}$ and $|F (z)| < 1$. That means that 
$|v (z)| < \eps / (2 C) < \eps$, and we have proved (a). 

By definition of $\delta_0$, (b) is satisfied for all $h \in (0, \beta]$. It will be satisfied as well for $h \in (\beta, 1]$ once we have proved (c).

Let us prove (c). Take $\beta < h \leq 1$. Since $|F (z)| \leq 1$ for $z \in S (\xi, h) \setminus S (\omega, r)$, we have:
\begin{align*}
\int_{S (\xi, h)} |v|^2 \, d \nu 
& \leq \int_{S (\xi, h) \setminus S (\omega, r)} \Big( \frac{\eps}{2C}\Big)^2 \, d \nu 
+ \Big( \frac{\eps}{2 C}\Big)^2 \int_{S (\omega, r)} |F|^{2 \delta_0} \, d \nu \\
& \leq \Big( \frac{\eps}{2 C}\Big)^2 \, \nu [S (\xi, h)] + \Big( \frac{\eps}{2C}\Big)^2 r^\gamma \Phi (\delta_0) \\
& \leq  \Big( \frac{\eps}{2 C}\Big)^2 \, C \, h^\gamma + r^\gamma 
= \Big( \frac{\eps}{2 C}\Big)^2 \, C \, h^\gamma + \beta^\gamma \frac{\eps^2}{2\ }\\
& \leq \Big( \frac{\eps}{2 C}\Big)^2 \, C \, h^\gamma + h^\gamma \frac{\eps^2}{2\ } 
\leq \bigg( \frac{\eps^2}{4 \ } + \frac{\eps^2}{2 \ } \bigg) \, h^\gamma \leq \eps^2 h^\gamma \, .
\end{align*}

Finally, by definition of $\delta_0$, there exist $t \in (0, \beta]$ and $\zeta \in \partial \D$ such that
\begin{displaymath}
\frac{1}{t^\gamma} \int_{S (\zeta, t)} |v|^2 \, d \nu = \Big( \frac{\eps}{2 C} \Big)^2 \frac{1}{t^\gamma} \int_{S (\zeta, t)} |F|^{2 \delta_0} \, d\nu 
\geq \Big( \frac{3}{4}\Big)^2 \, , 
\end{displaymath}
and (d) if proved.
\end{proof}
\begin{proof} [Proof of Theorem~\ref{theo clef}] 
Consider a sequence $(\eps_n)_{n \geq 1}$ of positive numbers such that $\sum_{n = 1}^\infty \eps_n < 1 / 4$. 
\smallskip

Using Proposition~\ref{prop 3}, we are going to construct by induction four sequences $(v_n)_n$ in $H^\infty$, $(\beta_n)_n$, with $\beta_1 = 1$, and 
$(t_n)_n$ in $(0, 1]$, and $(\zeta_n)_n$ in $\partial \D$ such that, for all $n \geq 1$:
\goodbreak
\begin{itemize}
\setlength \itemsep{-0.1 em}

\item [{\rm (S\,1)}] $\beta_n \geq t_n > \beta_{n + 1} \geq t_{n + 1}$;

\item [{\rm (S\,2)}] $|v_n (z)| < \eps_n$ for $|z| < 1 - \beta_n$;

\item [{\rm (S\,3)}] for all $\xi \in \partial \D$ and all $h \in (0, \beta_{n + 1}] \cup [\beta_n, 1]$:
\begin{displaymath}
\frac{1}{h^\gamma} \int_{S (\xi, h)} |v_n|^2 \, d \nu \leq \eps_n^2 \, ; \phantom{blablabla} 
\end{displaymath}

\item [{\rm (S\,4)}] $\dis \frac{1}{h^\gamma} \int_{S (\xi, h)} |v_n|^2 \, d \nu \leq 1$ for all $h \in (\beta_{n + 1}, \beta_n]$ and all $\xi \in \partial \D$;

\item [{\rm (S\,5)}] $\dis \frac{1}{t_n^\gamma} \int_{S (\zeta_n, t_n)} |v_n|^2 \, d \nu \geq \Big( \frac{3}{4} \Big)^2$, 
\end{itemize}
and
\begin{itemize}
\item [{\rm (S\,6)}] $\dis \lim_{n \to \infty} \beta_n = \lim_{n \to \infty} t_n = 0$.
\end{itemize}

Take $\beta_1 = 1$. With $\beta = \beta_1$ and $\eps = \eps_1$, let $v_1 = v$ be the function given by Proposition~\ref{prop 3} and 
$\zeta_1 =\zeta$ and $t_1 = t \leq \beta_1$ the numbers given by part (d) of that proposition. By Proposition~\ref{prop 3} (b) and (d) respectively, 
conditions {\rm (S\,4)} and {\rm (S\,5)} are satisfied for $n = 1$. Condition {\rm (S\,2)} is void for $n = 1$. For condition {\rm (S\,3)}, note that 
since $\nu$ is a vanishing $\gamma$-Carleson measure, there exists $\beta_2 > 0$ such that
\begin{displaymath}
\frac{\nu [ S (\xi, h)]}{h^\gamma} \leq \eps_1^2 (1 + \|v_1\|_\infty^2)^{- 1} 
\end{displaymath}
for all $h \in (0, \beta_2]$ and all $\xi \in \partial \D$. This implies, for these $h$'s and $\xi$'s:
\begin{displaymath}
\frac{1}{h^\gamma} \int_{S (\xi, h)} |v_1|^2 \, d \nu \leq \frac{\|v_1 \|_\infty^2 \, \nu [ S (\xi, h)]}{h^\gamma} \leq \eps_1^2 \, .
\end{displaymath}
It follows, with (c) of Proposition~\ref{prop 3}, that {\rm (S\,3)} is satisfied for $n = 1$. 

We can of course ask that $\beta_2 \leq 1/2$. 

Now, assume that $v_1, \ldots, v_{n + 1}$, $\beta_1, \ldots, \beta_{n + 1}$, $t_1, \ldots, t_{n + 1}$ and $\zeta_1, \ldots, \zeta_{n + 1}$ satisfying 
{\rm (S\,1)}, {\rm (S\,2)}, {\rm (S\,3)}, {\rm (S\,4)} and {\rm (S\,5)} have been constructed. 

As above, since $\nu$ is a vanishing $\gamma$-Carleson measure, there exists a positive number $\beta_{n + 2} \leq \min (\beta_{n + 1}, 1 / (n + 2))$ such that
\begin{displaymath}
\frac{\nu [ S (\xi, h)]}{h^\gamma} \leq \eps_n^2 (1 + \| v_{n + 1} \|_\infty^2)^{- 1} 
\end{displaymath}
for all $h \in (0, \beta_{n + 2}]$ and all $\xi \in \partial \D$. Using Proposition~\ref{prop 3} with $\beta = \beta_{n + 2}$ and $\eps  =\eps_{n + 1}$, 
we get $v_{n + 2} = v \in H^\infty$, $\zeta_{n + 2} = \zeta \in \partial \D$ and $t_{n + 2} = t \in (0, \beta_{n + 2}]$ and the induction step follows. 
\smallskip

We now set:
\begin{displaymath}
\quad  u (z) = \sum_{n = 1}^\infty v_n (z) \, , \quad z \in \D \, . 
\end{displaymath}
Thanks to {\rm (S\,2)}, this series converges uniformly on compact subsets of $\D$, so $u$ is analytic in $\D$.  

Take $\xi \in \partial\D$ and $h \in (0, 1]$. There exists a unique $n \geq 1$ such that $h \in (\beta_{n + 1}, \beta_n]$. By the triangle inequality:
\begin{align*}
\bigg( \frac{1}{h^\gamma} \int_{S (\xi, h)} |u|^2 \, d\nu \bigg)^{1/2} 
\leq \sum_{\substack{k = 1 \\ k \neq n}}^\infty \bigg( \frac{1}{h^\gamma} \int_{S (\xi, h)} & |v_k|^2 \, d\nu \bigg)^{1/2} \\ 
& + \bigg( \frac{1}{h^\gamma} \int_{S (\xi, h)} |v_n|^2 \, d\nu \bigg)^{1/2} \, .
\end{align*}
Now, since:
\begin{equation} \label{remarque utile}
\qquad  h \in (0, \beta_{k + 1}] \cup (\beta_k, 1] \quad \text{for every } k \neq n \, ,
\end{equation} 
we get, by {\rm (S\,3)} and {\rm (S\,4)}:
\begin{displaymath}
\bigg( \frac{1}{h^\gamma} \int_{S (\xi, h)} |u|^2 \, d\nu \bigg)^{1/2} \leq \bigg( \sum_{k \neq n} \eps_k \bigg) 
+ 1 \leq \frac{1}{4} + 1 = \frac{5}{4} \, \cdot
\end{displaymath}
Consequently, $u$ satisfies $(ii)$ of Theorem~\ref{theo clef}. Then condition $(ii)$ implies $(i)$ because $\nu$ is a finite measure, $u$ is 
bounded on $(1/2) \, \D$, and $\D \setminus (1/2)\, \D$ can be covered by a finite number of boxes $S (\xi, 1)$, with $\xi \in \partial \D$. 
\smallskip

To obtain $(iii)$, we use \eqref{remarque utile} again, with $h = t_n$, to get:
\begin{align*}
\bigg(\frac{1}{t_n^\gamma} \int_{S (\zeta_n, t_n)} |u|^2 \, d \nu \bigg)^{1/2} 
& \geq \bigg(\frac{1}{t_n^\gamma} \int_{S (\zeta_n, t_n)} |v_n|^2 \, d \nu \bigg)^{1/2} \\
& \qquad \qquad \qquad - \sum_{k \neq n} \bigg(\frac{1}{t_n^\gamma} \int_{S (\zeta_n, t_n)} |v_k|^2 \, d \nu \bigg)^{1/2} \\
& \geq \frac{3}{4} - \sum_{k \neq n} \eps_k \geq \frac{3}{4} - \frac{1}{4} = \frac{1}{2} \, \raise 1 pt \hbox{,} 
\end{align*}
and we have $(iii)$.
\end{proof}
%
\section {Hilbert-Schmidt regularization} \label{section HS}

We remarked in Section~\ref{section compactif} that if $\phi (z) = \frac{1 + z}{2}\,$, then $C_\phi$ is compactifiable on ${\mathfrak B}^2_\alpha$ 
by a weight in $H^\infty$. Actually, since $|\phi (\e^{it})| = \cos (t/2)$, we have 
$\int_{- \pi}^\pi \log \frac{1}{1 - |\phi (\e^{it}) |} \, dm (t) < \infty$, 
and \cite[Theorem~4.1]{LLQR-compactification} tells that the composition operator $C_\phi$ can be weighted to have a Hilbert-Schmidt operator on $H^2$; 
a fortiori, this weighted composition operator is Hilbert-Schmidt on ${\mathfrak B}^2_\alpha$ (see \cite[Theorem~3.12]{LLQR-comparison}). 
We can be more specific on an example, but unfortunately this example shows no difference between the Hardy and Bergman spaces.
\begin{proposition}\label{rate}
 Let $\phi \colon \D \to \D$ be defined by $\varphi(z) = \frac{1 + z}{2}$, and let $w (z) = (1 - z)^{\beta}$ with $\beta > - 1/2 $, so that $w \in H^2$.  
Then the  weighted composition operators $M_w C_\phi \colon H^2 \to H^2$ and  $M_w C_\phi \colon {\mathfrak B}^2 \to {\mathfrak B}^2$ are 
Hilbert-Schmidt if and only if $\beta > 1/2$. 
\end{proposition} 
\begin{proof}
The first item was proved in \cite[Proposition~2.4]{GDHL}. For the second item, we have to determine those $\beta$ such that
\begin{displaymath} 
\int_{\D} \frac{|w (z)|^2}{(1 - |\varphi (z)|^2)^2} \, dA (z) < \infty \, . 
\end{displaymath} 
Since $|\varphi(z)|$ approaches $1$ only when $z$ approaches $1$, we can as well consider
\begin{displaymath} 
I := \int_{\Delta} \frac{|w (z)|^2}{(1 - |\varphi(z)|^2)^2} \, dA (z) \, ,
\end{displaymath} 
where $\Delta = \D\cap D (1, 1)$. Passing in polar coordinates centered at $1$, we write, for $z\in \Delta$: $z = 1 - r \, \e^{i \theta}$ with $|\theta| < \pi/2$ 
and  $r < 2 \cos \theta$. Then, $\big|\frac{1 + z}{2} \big|^2 = 1 + \frac{r^2}{4} - r\cos \theta$ and
\begin{displaymath} 
I = \int_{- \pi/2}^{\pi/2} \int_{0}^{2 \cos \theta} \!\!\! \frac{r^{2 \beta + 1}}{r^2 (\cos \theta - r/4)^2} \, dr \, \frac{d\theta}{\pi} 
= \frac{2}{\pi} \int_{0}^{\pi/2} \int_{0}^{2 \cos \theta} \!\!\! \frac{r^{2 \beta - 1}}{(\cos \theta - r/4)^2} \, dr \, d\theta \, .
\end{displaymath} 
Making the change of variable $r = 2 t \cos \theta$, $0 \leq t \leq 1$ in the inner integral and observing that $1 \geq 1 - t/2 \geq 1/2$, we see that
\begin{displaymath} 
I \approx \int_{0}^{\pi/2} \int_{0}^{1} \frac{t^{2 \beta - 1} (\cos \theta)^{2 \beta}}{\cos^2 \theta} \, dt \, d\theta
= \bigg( \int_{0}^{1} t^{2 \beta - 1} \, dt \bigg) \, \bigg(\int_{0}^{\pi/2} (\sin \theta)^{2 \beta - 2} \, d\theta \bigg) \, . 
\end{displaymath} 
So, clearly, $I < \infty$ if and only if $\beta > 1/2$. 
\end{proof}

Fortunately, examples showing the difference between Hardy and Bergman spaces exist.
\goodbreak
\begin{theorem} \label{lent}
There exists a Blaschke product $B$ which can be Hilbert-Schmidt regularized, and more, on $\mathfrak{B}^2$, but not on $H^2$. 
\end{theorem} 
\begin{proof}
Any Blaschke product $B$ is an inner  function, i.e. $|B^\ast| = 1$ $m$-almost everywhere on the unit circle, implying, by 
\cite[Theorem~3.1]{LLQR-compactification}, that $M_w C_B$ is compact on $H^2$ for no weight $w \in H^2$, with $w \neq 0$. 

On the other hand, as a consequence of \cite[Theorem~3.1]{LELIQUR}, we proved (\cite[Theorem~4.4]{PDHL}; see also \cite[Theorem~13]{LD}) that there 
exist Blaschke products $B$ (which we called \emph{slow Blaschke products}) such that $C_B$ is compact on the Bergman-Orlicz space 
$\mathfrak{B}^{\Psi_2}$, and hence belong to every Schatten class $S_p$ of $\mathfrak{B}^2$.  
\end{proof}

Moreover, we can give the  following quantitative precision to Theorem~\ref{lent}.

\begin{theorem} \label{subex} 
For any sequence $(\eps_n)$ of positive numbers with limit zero, there is a Blaschke product $B$ such that
\begin{displaymath} 
a_{n} (C_B \colon \mathfrak{B}^2 \to \mathfrak{B}^2) \lesssim \e^{- n \eps_n} \, .
\end{displaymath} 
\end{theorem} 
\begin{proof} 
We can assume that $\eps_n$ decreases and that $n \eps_n \uparrow \infty$ with $n \eps_n \geq \sqrt{n}$. 

For a given symbol $\varphi$, we set
\begin{displaymath} 
\chi (h) = A (\{z \tq |\varphi (z) | \geq 1 - h \} ) \, . 
\end{displaymath} 
We use \cite[Theorem 5.1]{LQR} which implies that 
\begin{equation} \label{conseq} 
a_{n} (C_\varphi) \lesssim \inf_{0 < h < 1} \bigg[\sqrt{n} \, \e^{- n h} + \sqrt{\frac{\chi (h)}{h^2}} \bigg] \, .
\end{equation}
Let $\delta \colon (0, 1) \to (0, 1)$ be a non-increasing and piecewise linear map, decreasing to $0$ so slowly at the origin that
\begin{displaymath} 
\delta (1 - |z|) \leq 4 \, \eps_n \quad \Longrightarrow \quad 1 - |z|\leq \eps_{n}^{2} \exp (- 2 \, n \, \eps_n) \, .
\end{displaymath} 
By \cite[Theorem~3.1]{LELIQUR} again, there exists a Blaschke product $B$ such that $|B (z)| \leq \exp \big(- \delta (1 - |z|) \big)$.
Take $\varphi = B$ in \eqref{conseq} and observe that, for $h = 2 \, \eps_n \leq 1/2$, we have
\begin{displaymath} 
|B (z)| \geq 1 - h \quad \Longrightarrow \quad \exp \big(- \delta (1 - |z|) \big) \geq 1 - h \geq \exp ( - 2 h) \, .
\end{displaymath} 
Hence
\begin{displaymath} 
\delta (1 - |z|) \leq 4 \, \eps_n\quad \text{and} \quad 1 - |z| \leq \eps_{n}^{2}\exp (- 2 \, n \, \eps_n) 
\end{displaymath} 
and
\begin{displaymath} 
\chi (h) \leq 2 \, \eps_{n}^{2} \exp (- 2 \, n \, \eps_n) \, .
\end{displaymath} 
Inserting this in \eqref{conseq}, we get the result.
\end{proof}

In order to find a necessary and sufficient condition for a symbol can be weighted in a Hilbert-Schmidt operator, we make some observations.
\smallskip

As recalled, a weighted composition operator which is Hilbert-Schmidt on $H^2$ is also Hilbert-Schmidt on ${\mathfrak B}^2$. We know that 
$C_\phi$ is Hilbert-Schmidt on $H^2$ if and only if 
\begin{displaymath} 
\int_\T \frac{1}{1 - |\phi|^2} \, dm < \infty \, . 
\end{displaymath} 
Equivalently:
\begin{equation} 
\sum_{n = 0}^\infty \| \phi^n \|_{H^2}^2 < \infty \, .
\end{equation} 
On the other hand, $C_\phi$ can be weighted to become a Hilbert-Schmidt operator on $H^2$ if and only if 
\begin{displaymath} 
\int_\T \log \frac{1}{1 - |\phi|} \, dm < \infty \,, 
\end{displaymath} 
(\cite[Theorem~4.1]{LLQR-compactification}), which is equivalent to 
\begin{equation} \label{CNS Hardy}
\sum_{n = 0}^\infty \frac{1}{n + 1} \, \| C_\phi (e_n)\|_{H^2}^2 < \infty \, .
\end{equation} 

Now, writing $e_n (z) = z^n$, and since $\big( (n + 1)^{(\alpha + 1)/2} \, e_n \big)_n$ is an orthonormal basis of ${\mathfrak B}_\alpha^2$, $C_\phi$ is 
Hilbert-Schmidt on ${\mathfrak B}_\alpha^2$ if and only if
\begin{displaymath} 
\sum_{n = 0}^\infty (n + 1)^{\alpha  + 1} \| \phi^n \|_{{\mathfrak B}_\alpha^2}^2  = 
\sum_{n = 0}^\infty \| C_\phi \big(  (n + 1)^{(\alpha + 1)/2} \, e_n \big) \|_{{\mathfrak B}_\alpha^2}^2 < \infty \, .
\end{displaymath} 
By comparison with \eqref{CNS Hardy}, we might think that $C_\phi$ can be weighted to become a Hilbert-Schmidt operator on ${\mathfrak B}_\alpha^2$ 
if and only if
\begin{equation} 
\sum_{n = 0}^\infty (n + 1)^\alpha  \| \phi^n \|_{{\mathfrak B}_\alpha^2}^2 < \infty \, .
\end{equation} 
Since 
\begin{equation} 
\sum_{n = 0}^\infty (n + 1)^\alpha  \| \phi^n \|_{{\mathfrak B}_\alpha^2}^2 
= \int_\D \frac{(1 - |z|^2)^\alpha }{(1 - |\phi (z) |^2)^{\alpha + 1}} \, dA  (z) \, ,
\end{equation} 
this guesswork takes the following form: is tit true that there exists a weight $w$ such that 
$M_w C_\phi \colon {\mathfrak B}_\alpha^2 \to {\mathfrak B}_\alpha^2$ is Hilbert-Schmidt if and only if 
\begin{equation} \label{conjecture HS}
\quad \int_\D \frac{(1- |z|^2)^\alpha}{(1 - |\phi (z) |^2)^{\alpha + 1}} \, dA (z) < \infty \quad ?
\end{equation} 

We do not know if \eqref{conjecture HS} implies the existence of a weight $w \not\equiv 0$ for which $M_w C_\phi$ is Hilbert-Schmidt, but, in any case, 
it implies that $C_\phi$ is compactifiable. 
\goodbreak

\begin{proposition} \label{interesting prop}
If, for $\alpha > - 1$, we have $\dis \int_\D \frac{(1 - |z|^2)^\alpha}{(1 - |\phi (z) |)^{\alpha  + 1} }\, dA (z) < \infty$, then 
\begin{displaymath}
\qquad\qquad \qquad \quad \lim_{z \to \xi} \frac{1 - |z|}{1 - |\phi (z)|} = 0 \qquad \text{for almost all } \xi \in \T \, . 
\end{displaymath}
\end{proposition}

Recall that, by Theorem~\ref{compactif}, this last condition means that $C_\phi$ is compactifiable on ${\mathfrak B}_\alpha^2$.
\begin{proof}
We set:
\begin{displaymath} 
g (z) = \bigg( \frac{1 - |z|^2}{1 - |\phi (z)|^2} \bigg)^{\alpha + 1} \, \cdot
\end{displaymath} 
Let $r_n = 1 - 2^{- n}$ and:
\begin{displaymath} 
\Gamma_n =\{ z \in \D \tq r_n \leq |z| < r_{n + 1} \} \, .
\end{displaymath} 

\hskip -1.5 pt Now, $1 / (1 - |\phi|^2)^{\alpha + 1}$ is subharmonic (and even logarithmically-subharmonic), because we can write 
$1 / (1 - |\phi|)^{\alpha + 1} = \sum_{k = 0}^\infty c_k (\alpha) |\phi|^{2 k}$ with $c_k (\alpha) \geq 0$. Hence, we have, since 
$A (\Gamma_n) \approx 1 - r_n^2$:
\begin{align*}
\int_\T g (r_n \e^{i \theta}) \, d\theta 
& = \int_\T \frac{(1 - r_n^2)^{\alpha + 1}}{(1 - |\phi (r_n \e^{i \theta}) |^2)^{\alpha  + 1} } \, d\theta \\
& \lesssim \frac{1}{1 - r_n^2} \int_{\Gamma_n} \frac{(1 - r_n^2)^{\alpha  + 1}}{(1 - |\phi (z) |^2)^{\alpha  + 1}} \, dA (z) \\
& \lesssim \int_{\Gamma_n} \frac{(1 - |z|^2)^\alpha}{(1 - |\phi (z) |^2)^{\alpha + 1}} \, dA (z) 
\end{align*}
(we used that $1 - r_n^2 \leq 2 (1 - r_n) = 2\times 2^{- n} = 4 (1 - r_{n + 1}) \leq 4 (1 - |z|) \leq 4 (1 - |z|^2)$ for $z \in \Gamma_n$, so 
$( 1 - r_n^2)^\alpha \lesssim (1 - |z|^2)^\alpha$ when $\alpha \geq 0$, and, when $- 1 < \alpha < 0$, we used that $1 - |z|^2 \leq 1 - r_n^2$ for 
$z \in \Gamma_n$). The sets $\Gamma_n$ being disjoint, we get that: 
\begin{align*} 
\int_\T \bigg( \sum_{n = 0}^\infty g (r_n \e^{i \theta}) \bigg) \, d \theta
& = \sum_{n = 0}^\infty \int_\T g (r_n \e^{i \theta}) \, d\theta \\
& \lesssim \sum_{n = 0}^\infty \int_{\Gamma_n} \frac{(1 - |z|^2)^\alpha}{(1 - |\phi (z) |^2)^{\alpha  + 1}} \, dA (z) \\
& = \int_\D \frac{(1 - |z|^2)^\alpha}{(1 - |\phi (z) |^2 )^{\alpha + 1} } \, dA (z) < \infty \, ,
\end{align*} 
meaning that the function $\sum_{n = 0}^\infty g (r_n \, \cdot)$ is integrable on $\T$. It follows that $g (r_n \, \cdot) \converge_{n \to \infty} 0$ 
almost everywhere. Since the existence of a radial limit implies that of an angular limit, we obtain that 
$\angle \lim_{z \to \xi} \frac{1 - |z|}{1 - |\phi (z)|} = 0$ for almost all $\xi \in \T$. By the Julia-Caratheodory theorem, it follows that 
$\lim_{z \to \xi} \frac{1 - |z|}{1 - |\phi (z)|} = 0$ for almost all $\xi \in \T$. 
\end{proof}

An a priori different condition than \eqref{conjecture HS} appears in the following theorem.
\goodbreak

\begin{theorem} \label{theo HS suff-ness} 
Let $d \lambda_\alpha (r) = 2 \, (\alpha + 1) \, (1 - r^2)^\alpha \, r \, dr$, be the marginal probability measure on $[0, 1)$ of  $d A_\alpha$,  and
\begin{equation}
G (\theta) =  \int_0^1 \frac{d \lambda_\alpha (r)}{(1 - |\phi (r \e^{i \theta})|^2)^{\alpha + 2} } \, \cdot
\end{equation}
Then:
\smallskip

$1)$ If\/ $\log G \in L^1 (0, 2 \pi)$, then there exists $w \in H^\infty$, $w \not\equiv 0$, such that $M_w C_\phi$ is Hilbert-Schmidt on 
${\mathfrak B}_\alpha^2$.
\smallskip

$2)$ Conversely, if there exists such a weight $w$, then $\log G \in L^{1, \infty} (0, 2 \pi)$. 
\end{theorem} 

Note that $G \geq 1$, so $\log G \geq 0$. 
\smallskip

Recall that $L^{1, \infty} (\mu)$ is the space of (classes of) measurable functions $f$ such that 
$\sup_{a > 0} a \, m (\{ |f| > a \}) < \infty$, and that $L^1 (\mu) \subseteq L^{1, \infty} (\mu)$, by Markov's inequality.

\subsection{Proof of $\bf 1)$ of Theorem~\ref{theo HS suff-ness}} 
For convenience, we set
\begin{equation} \label{def U}
U (z) = \frac{1 \ \ }{(1 - | \phi (z) |^2)^{\alpha + 2}} \, \cdot
\end{equation}
We will use two lemmas. For that, we denote $\rho$ the pseudo-hyperbolic metric on $\D$. Recall that
\begin{displaymath}
\qquad \quad \rho (u, v) = \bigg| \frac{u - v}{1 - \bar{u} v} \bigg| \, , \quad u, v \in \D \, .
\end{displaymath}
\begin{lemma} \label{lemma L1}
There is a positive constant $C = C (\alpha)$ such that, for $u, v \in \D$:
\begin{equation}
\rho (u, v) \leq \frac{1}{2} \quad \Longrightarrow \quad \frac{1}{C} \leq \frac{U (u)}{U (v)} \leq C \, .
\end{equation}
\end{lemma}
\begin{proof}
Since 
\begin{displaymath}
\bigg( \frac{1}{2} \times \frac{1 - | \phi (v) |}{1 - | \phi (u) |} \bigg)^{\alpha  + 2} \leq \frac{U (u)}{U (v)} 
\leq \bigg( 2 \times \frac{1 - | \phi (v) |}{1 - | \phi (u) |} \bigg)^{\alpha  + 2} \, ,
\end{displaymath}
it suffices to show that there is a positive constant such that 
\begin{displaymath}
\frac{1}{C} \leq \frac{1 - | \phi (v) |}{1 - | \phi (u) |} \leq C 
\end{displaymath}
when $\rho (u, v) \leq 1/2$. Moreover, by the Schwarz-Pick inequality, we have:
\begin{displaymath}
\rho \big( |\phi (u)|, |\phi (v)| \big) \leq \rho \big( \phi (u), \phi (v) \big) \leq \rho (u, v) \, ,
\end{displaymath}
it suffices to majorize $q := \frac{1 - a}{1 - b}$ when $\rho (a, b) \leq 1/2$ and $0 \leq a, b < 1$ (the minoration will come by exchanging $a$ and $b$). 

If $a \geq b$, then $q \leq 1$

If $a < b$, we remark that $\rho (a, b) \leq 1/2$ writes $T_a (b) := \frac{a - b}{1 - a b} \geq - 1 / 2$. Since $T_a$ is decreasing on $[- 1, 1]$, we get 
$b \leq T_a (- 1 / 2)$, i.e. $b \leq \frac{1 + 2 a}{2 + a}$, and $1 - b \geq \frac{1 - a}{2 + a}$. Therefore $q \leq 2 + a \leq 3$.
\end{proof}

Let, for $n \geq 0$: 
\begin{displaymath}
r_n = \exp (- 2^{- n})   
\end{displaymath}
and 
\begin{displaymath}
\Gamma_n = \{z \in \D \tq r_n \leq |z| < r_{n + 1} \} \, .
\end{displaymath}
\begin{lemma} \label{lemma L2}
For $r_n \leq u, v \leq r_{n + 1}$, we have $\rho (u \, \e^{i \theta}, v \, \e^{i \theta} ) \leq 1/2$, for every $\theta \in \R$.
\end{lemma}
\begin{proof}
It is a simple computation:
\begin{displaymath}
\rho (u \, \e^{i \theta}, v \, \e^{i \theta}) = \frac{|u - v|}{1 - u v} \leq \frac{r_{n + 1} - r_n}{1 - r_{n + 1}^2} 
= \frac{r_{n + 1} - r_{n + 1}^2}{1 - r_{n + 1}^2} 
= \frac{r_{n + 1}}{1 + r_{n + 1}} \leq \frac{1}{2} \, \cdot \qedhere
\end{displaymath}
\end{proof}

Now, we can finish the proof of $1)$ of Theorem~\ref{theo HS suff-ness}.
\smallskip
 
We have to show that there exists a non-null function $w_0 \in H^\infty$ such that
\begin{equation} \label{to get}
\int_\D |w_0|^2 U \, dA_\alpha < \infty \, ,
\end{equation}
where $dA_\alpha  (z) = (\alpha + 1) (1 - |z|^2)^\alpha \, dA (z)$.

For every $w \in H^\infty$, we have:
\begin{displaymath}
\int_\D |w|^2 U \, dA_\alpha 
= \int_{D (0, \e^{- 1})} |w|^2 U \, dA_\alpha + \sum_{n = 0}^\infty \int_{\Gamma_n} |w|^2 U \, dA_\alpha 
\end{displaymath}
For every $n \geq 0$:
\begin{displaymath}
\int_{\Gamma_n} |w|^2 U \, dA_\alpha 
= 2 \, (\alpha  + 1) \int_{r_n}^{r_{n + 1}} 
\bigg( \frac{1}{2 \pi} \int_0^{2 \pi} | w (r \e^{i \theta}) |^2 \, U (r \e^{i \theta}) \, d \theta  \bigg) (1 - r^2)^\alpha \, r \, dr \\ 
\end{displaymath}
As said in the proof of Proposition~\ref{interesting prop}, $U$ is logarithmically-subharmonic; hence  
the function $|w|^2 \, U$ is also logarithmically-subharmonic; in particular, it is subharmonic; so we have (see \cite[Theorem~1.6, page~9]{Duren}), 
for $r_n \leq r \leq r_{n + 1}$:
\begin{displaymath}
\frac{1}{2 \pi} \int_0^{2 \pi} | w (r \e^{i \theta}) |^2 \, U (r \e^{i \theta}) \, d \theta 
\leq \frac{1}{2 \pi} \int_0^{2 \pi} | w (r_{n + 1} \e^{i \theta}) |^2 \, U (r_{n + 1} \e^{i \theta}) \, d \theta \, .
\end{displaymath}
By Lemma~\ref{lemma L1} and Lemma~\ref{lemma L2}, we have $U (r_{n + 1} \e^{i \theta}) \leq C \, U (r_n \, \e^{i \theta})$. 
But $r_n = r_{n + 1}^2$, so $U (r_{n + 1} \e^{i \theta}) \leq C \, U (r_{n + 1}^2 \, \e^{i \theta})$, and hence
\begin{displaymath}
\frac{1}{2 \pi} \int_0^{2 \pi} | w (r \e^{i \theta}) |^2 \, U (r \e^{i \theta}) \, d \theta 
\leq C \, \frac{1}{2 \pi} \int_0^{2 \pi} | w (r_{n + 1} \e^{i \theta}) |^2 \, U (r_{n + 1}^2 \e^{i \theta}) \, d \theta \, .
\end{displaymath}
By the subharmonicity of $|w|^2 \, U$ again, we obtain:
\begin{displaymath}
\frac{1}{2 \pi} \int_0^{2 \pi} | w (r \e^{i \theta}) |^2 \, U (r \e^{i \theta}) \, d \theta 
\leq C \, \frac{1}{2 \pi} \int_0^{2 \pi} | w ( \e^{i \theta}) |^2 \, U (r_{n + 1} \e^{i \theta}) \, d \theta \, .
\end{displaymath}
Using Lemma~\ref{lemma L1} and Lemma~\ref{lemma L2} again, we have, for every $r_n \leq r < r_{n + 1}$:
\begin{displaymath}
\frac{1}{2 \pi} \int_0^{2 \pi} | w ( \e^{i \theta}) |^2 \, U (r_{n + 1} \e^{i \theta}) \, d \theta 
\leq C\, \frac{1}{2 \pi} \int_0^{2 \pi} | w ( \e^{i \theta}) |^2 \, U (r \, \e^{i \theta}) \, d \theta \, .
\end{displaymath}
Therefore:
\begin{displaymath}
\int_{\Gamma_n} |w|^2 U \, dA_\alpha 
\leq C^{\, 2} \int_{r_n}^{r_{n + 1}} 
\bigg( \frac{1}{2 \pi} \int_0^{2 \pi} | w ( \e^{i \theta}) |^2 \, U (r \, \e^{i \theta}) \, d \theta \bigg)  \, d\lambda_\alpha ( r ) \, .
\end{displaymath}
Using the Fubini theorem, we finally obtain:
\begin{align*}
\int_\D |w|^2 U \, dA_\alpha 
& \leq \int_{D (0, \e^{- 1})} |w|^2 U \, dA_\alpha \\
&  \qquad \quad + C^{\, 2} \frac{1}{2 \pi} \int_0^{2 \pi} 
\bigg(\int_{\e^{- 1}}^1 U (r\, \e^{i \theta}) \,  d \lambda_\alpha ( r ) \bigg) \, |w (\e^{i \theta}) |^2 \, d \theta \, .
\end{align*}
Since
\begin{displaymath}
G (\theta) = \int_0^1 U (r\, \e^{i \theta}) \, d \lambda_\alpha ( r ) \, ,
\end{displaymath}
we have:
\begin{equation} \label{formula}
\int_\D |w|^2 U \, dA_\alpha \leq 
\int_{D (0, \e^{- 1})} |w|^2 U \, dA_\alpha + C^{\, 2} \frac{1}{2 \pi} \int_0^{2 \pi} G (\theta) \, |w (\e^{i \theta}) |^2 \, d \theta \, .
\end{equation}

We now use Szeg\"o's theorem (see \cite[Theorem~3.1, Chapter~IV, page~139]{Garnett}, or \cite[Section 8.3]{Rudin}):
\begin{align*}
\inf_{w \in H^\infty, w (0) = 1} \frac{1}{2 \pi} \int_0^{2 \pi} |w (\e^{i \theta})|^2 G (\theta) \, d \theta 
= \exp \bigg( \frac{1}{2 \pi} \int_0^{2 \pi} \log G (\theta) \, d \theta \bigg) \, .
\end{align*}

Remarking that the hypothesis of the theorem writes:
\begin{displaymath}
\int_0^{2 \pi} \log G (\theta) \, d \theta < \infty \, ,
\end{displaymath}
that shows that there exists $w_0 \in H^\infty$ with $w_0 (0) = 1$ such that 
\begin{displaymath}
\frac{1}{2 \pi} \int_0^{2 \pi} |w_0 (\e^{i \theta})|^2 G (\theta) \, d \theta  < \infty \, .
\end{displaymath}
With \eqref{formula}, that shows that $w_0$ satisfies \eqref{to get}, and that ends the proof of $1)$ of 
Theorem~\ref{theo HS suff-ness}.
\smallskip

Note that the proof shows that we can actually get  a polynomial for $w_0$. 

\subsection{Proof of $\bf 2)$ of Theorem~\ref{theo HS suff-ness}}
We may, and do, assume that $\| w \|_\infty = 1$.
\smallskip

By hypothesis, we have
\begin{displaymath}
\int_\D |w (z)|^2 \frac{(1 - |z|^2)^\alpha \ }{\big( 1 - |\phi (z)|^2 \big)^{\alpha + 2}} \, dA (z) < \infty \, .
\end{displaymath}
Setting, with $U$ defined in \eqref{def U}:
\begin{equation}
\psi (\theta) = \int_{1/2}^1 |w (r \, \e^{i \theta}) |^2 \, U (r \, \e^{i \theta}) \, d \lambda_\alpha ( r ) \, ,
\end{equation}
we hence have $\psi \in L^1 (0, 2 \pi)$.

Let 
\begin{displaymath}
\tilde G (\theta) = \int_{1/2}^1 U (r\, \e^{i \theta}) \, d \lambda_\alpha ( r ) 
\end{displaymath}
and
\begin{displaymath}
J (\theta) = \inf \{ | w (r \, \e^{i \theta}) |^2 \tq 1/ 2 \leq r < 1 \} \, .
\end{displaymath}
We have $ \psi (\theta) \geq \tilde G (\theta) \, J (\theta)$, so
\begin{displaymath}
\log \tilde G \leq \log \psi + \log (1 / J) \leq \log^+ \psi + \log (1 / J) \leq \psi + \log (1 / J) \, .
\end{displaymath}
Since $U \geq 1$, we have $\tilde G (\theta) \geq C_\alpha$, with $C_\alpha = ( 3 / 4)^{\alpha  + 1} > 0$; hence 
$\log \tilde G (\theta) \geq \log C_\alpha > - \infty$. Therefore, to get $\log G \in L^{1, \infty} (0, 2 \pi)$ and finish the proof of $2)$ of 
Theorem~\ref{theo HS suff-ness}, it suffices to prove that $\log \tilde G \in L^{1, \infty} (0, 2 \pi)$, and for that, to prove that 
$\log ( 1 / J) \in L^{1, \infty} (0, 2 \pi)$. This is the object of the following theorem.

\begin{theorem} \label{theo with inf}
Let $v \in H^\infty$ such that $\| v \|_\infty = 1$ and set
\begin{equation}
I_v (\theta) = \inf \{ |v (r \, \e^{i \theta}) | \tq 1 / 2 \leq r < 1 \} \, .
\end{equation}
Then $\log ( 1 / I_v) \in L^{1, \infty} (0, 2 \pi)$.
\end{theorem}
\begin{proof}
We can write $v (z) = B (z) \, v_0 (z)$, where $B$ is the Blaschke product whose zeros are those of $v$, and $v_0$ does not vanish. Since 
\begin{displaymath} 
I_v \geq I_B \times I_{v_0} \, ,
\end{displaymath} 
it suffices to prove that $\log (1 / I_B) \in L^{1, \infty} (0, 2 \pi)$ and $\log (1 / I_{v_0} ) \in L^{1, \infty} (0, 2 \pi)$. 
\medskip

\noindent {\sl Case of a non vanishing function.} 

We can write $v_0 = \exp (- h)$, where $h \colon \D \to \{ \Re z > 0 \}$. We have $h = u  + i \tilde u$, where $u = \Re h$ and $\tilde u$ is the conjugate 
function of $u$. Since $u > 0$, $u = {\rm P} [\mu]$ is the Poisson integral of a positive measure $\mu$, and we have
\begin{displaymath} 
\qquad \qquad u (r \, \e^{i \theta}) \leq C \, M_\mu (\theta) \, \quad \forall r \in [0, 1) \, ,
\end{displaymath} 
where $M_\mu$ is the Hardy-Littlewood maximal function of $\mu$. Then:
\begin{displaymath} 
| v_0 (r \, \e^{i \theta} ) | \geq \exp \big( - C \, M_\mu (\theta)  \big) \, ;
\end{displaymath} 
so $I_{v_0} (\theta) \geq \exp \big( - C \, M_\mu (\theta)  \big)$, and $\log \big(1 / I_{v_0} (\theta) \big) \leq C \, M_\mu (\theta)$. 
Since $M_\mu \in L^{1, \infty} (0, 2 \pi)$, by Kolmogorov's theorem, we obtain that  $\log (1 / I_{v_0}) \in L^{1, \infty} (0, 2 \pi)$.
\medskip\goodbreak

\noindent {\sl Case of a Blaschke product.}
\smallskip

This case will follow from the next result. We note $\arg z$ the principal argument of $z$: $- \pi < \arg z \leq \pi$. 
\begin{proposition} \label{proposition Blaschke}
Let $B_0$ be a Blaschke product whose zeros $a_n$ have modulus greater or equal to some positive constant $c$, say $c = 3 / 4$. 
Then there exist $f \in L^1 (- \pi, \pi)$ and $u = {\rm P} [w]$, with $w \in L^1 (- \pi, \pi)$, such that
\begin{equation} 
\qquad \log \big( 1 / |B_0 (z)| \big) \leq f (\arg z) + u (z) \, , \quad \text{for all } z \in \D \, .
\end{equation} 
\end{proposition}

For $a \in \D$, we denote $\Delta (a , 1 /2 )$ the pseudo-hyperbolic disk of center $a$ and radius $1 / 2$. 
\smallskip

We begin by two lemmas.

\begin{lemma} \label{essential lemma}
For $a \in \D$, we set
\begin{displaymath}
\phi_a (z) = \frac{a - z}{1 - \bar{a} z} \, \raise 1 pt \hbox{,}
\end{displaymath}
as well as $I_a = I_{\phi_a}$ and $G_a = \log (1 / I_a)$. Then, for every $a \in \D$, we have $G_a \in L^1 (- \pi, \pi)$.
\end{lemma}

First, we have $G_a \geq 0$. Then:
\begin{displaymath}
| \phi_a (z) | = \frac{| a - z|}{| 1 - \bar{a} z |} \geq \frac{|z - a |}{2} \, ;
\end{displaymath}
so, it suffices to give a lower estimate of $|z - a |$.

We separate two cases.
\smallskip

$\bullet$ First case: $|a| \leq 1 / 4$. Then we have $| r \, \e^{i \theta} - a | \geq 1 / 4$ when $1 / 2 \leq r < 1$; hence $G_a (\theta) \leq \log 8$ 
for all $\theta$ and $G_a \in L^1 (- \pi, \pi)$.
\smallskip

$\bullet$ Second case: $|a| > 1 / 4$.  We can assume that $1 / 4 < a < 1$. If $z = r \, \e^{i \theta}$, then, for $|\theta| \leq \pi / 2$: 
\begin{displaymath}
| z -a | \geq {\rm dist}\, (a, R_\theta) = a \, |\sin \theta | \, , 
\end{displaymath}
where $R_\theta$ is the ray passing through $0$ and $\e^{i \theta}$, so 
\begin{displaymath}
G_a (\theta)  \leq \bigg| \log \bigg( \frac{a}{2} \, |\sin \theta| \bigg) \bigg| 
\end{displaymath}
and $G_a \in L^1 (- \pi, \pi)$.
\end{proof}
\begin{lemma} \label{essential lemma 2}
There is a positive constant $C$ such that, for $3 / 4 \leq a < 1$ and $h = 1 - a$, we have:
\begin{equation}
\qquad \qquad |\theta| \leq C h \quad \text{when } z = r \, \e^{i \theta} \in \Delta (a, 1 / 2) \, . 
\end{equation}
\end{lemma}
\begin{proof} 
The pseudo-hyperbolic disk $\Delta (a, 1 / 2)$ is equal to the Euclidean disk $D (\tilde a, R)$, with 
\begin{displaymath}
\tilde a = \frac{3}{4 - |a|^2} \, a \quad \text{and} \quad R = 2 \, \frac{1 - |a|^2}{4 - |a|^2}
\end{displaymath}
(see \cite[page~3]{Garnett}). For $0 < a < 1$ and $h = 1 - a$, we have $R \leq 4 \, h / 3$ and $3 \, a / 4 \leq \tilde a \leq a$. 
Hence $\Delta (a, 1 / 2)$ is contained in the angular sector of vertex $0$ and half-angle $\theta_a$ such that 
$\sin \theta_a = R / \tilde a \leq (4 \, h / 3) / (3 a / 4)$ (see Figure~\ref{figure}). For $3 / 4 \leq a < 1$, that gives $\sin \theta_a \leq (64 / 27) \, h$. 
It follows that there is $C > 0$ ($C = 64 \pi / 54$ works) such that $|\theta| \leq C \, h$ when $z = r \, \e^{i \theta} \in \Delta (a, 1 / 2)$. 
\end{proof}
\begin{proof} [Proof of Proposition~\ref{proposition Blaschke}] 
We restrict ourselves, for the time, to $3/4 \leq a < 1$.

Note that, since $3/ 4 \leq a < 1$ and $a = 1 - h$, we have $0 < h \leq 1 / 4$.
\smallskip

$\bullet$ Let $z = r \, \e^{i \theta} \in \Delta (a, 1 / 2)$. 

We write

\begin{displaymath}
|\phi_a (z) | = \frac{|z - a |}{\dis a \, \Big| z - \frac{1}{a} \Big|} \geq \frac{|z - a |}{\dis a \, \Big[ |z - a| + \Big(\frac{1}{a} - a \Big)\Big]} 
= \frac{1}{\dis a + \frac{1 - a^2}{|z - a |}} \, ;
\end{displaymath}
so, if $z \in \Delta (a, 1 / 2)$, and $z = r \, \e^{i \theta}$, we have $| \theta | < \pi / 2$; hence, when $\theta \neq 0$:
\begin{displaymath}
|\phi_a (z) | \geq \frac{1}{\dis a + \frac{1 - a^2}{a \, | \sin \theta | }} \, \cdot
\end{displaymath}
It follows from Lemma~\ref{essential lemma 2} that, for another constant $C$:
\begin{displaymath}
\frac{1}{I_a (\theta)} \leq a + \frac{1 - a^2}{a \, |\sin \theta|} \leq C \, \frac{h}{|\theta|} \, \raise 1 pt \hbox{,} 
\end{displaymath}
so
\begin{displaymath}
G_a (\theta) \leq \log \bigg(C \, \frac{h}{| \theta | } \bigg) 
\end{displaymath}
and 
\begin{displaymath}
\int_0^{C h} G_a (\theta) \, d\theta \leq  \bigg[ \theta \log \bigg(C \, \frac{h}{\theta} \bigg) + \theta \bigg]_0^{Ch} = C h \, .
\end{displaymath}

Setting, for $|\theta| \leq \pi$ (recall that $a = 1 - h$):
\begin{equation}
f_a (\theta) = \log \bigg(C \, \frac{h}{| \theta | } \bigg) \, \ind_{[- C h, C h]} (\theta) \, ,
\end{equation}
we hence have:
\begin{equation} \label{majo 1}
\qquad \log \bigg( \frac{1}{|\phi_a (z)|} \bigg) \leq f_a (\theta) \quad \text{ for } z = r \, \e^{i \theta} \in \Delta (a, 1 / 2) 
\end{equation}
(since then $|\theta| \leq C h$). Moreover, we have 

\begin{equation}
\| f_a \|_1 \leq 2 \, C h \, .
\end{equation}
\smallskip\goodbreak

$\bullet$ Now, let $z \in \D \setminus \Delta (a, 1 / 2)$. 
\smallskip

Let $D_a$ be the (Euclidean) disk of diameter $[c, 1 / c]$, where $c$ is the point of the segment $\partial \Delta (a, 1 / 2) \cap [\,0, 1)$ such that 
$0 < c < a$, and 
\begin{equation} 
A_a = \partial \D \cap D_a \, . 
\end{equation} 
We write simply $A_a = A$ thereafter. 
\smallskip

We set
\begin{equation}
w_a = 2 \log 2 \, \ind_A \quad \text{and} \quad u_a = {\rm P} [w_a] \, ,
\end{equation}
the Poisson integral of $w_a$.  

We have, for some  positive constant $C$:
\begin{equation}
\| w_a \|_1 = 2 \log 2 \, m (A) \leq C h \, .
\end{equation}
In fact, the diameter of $D_a$ is $\frac{1}{c} - c$ and $\frac{1}{2} = |\phi_a (c)| = \frac{a - c}{1 - a c}\,$, so
\begin{displaymath}
c = \frac{2 \, a - 1}{2 - a} = \frac{1 - 2 h}{1 + h} = 1 - 3 h + {\rm o}\, (h) \, ,
\end{displaymath}
and the diameter of $D_a$ is equal to $6 h + {\rm o}\, (h)$.

\begin{figure}[ht]
\centering
\includegraphics[width=10cm]{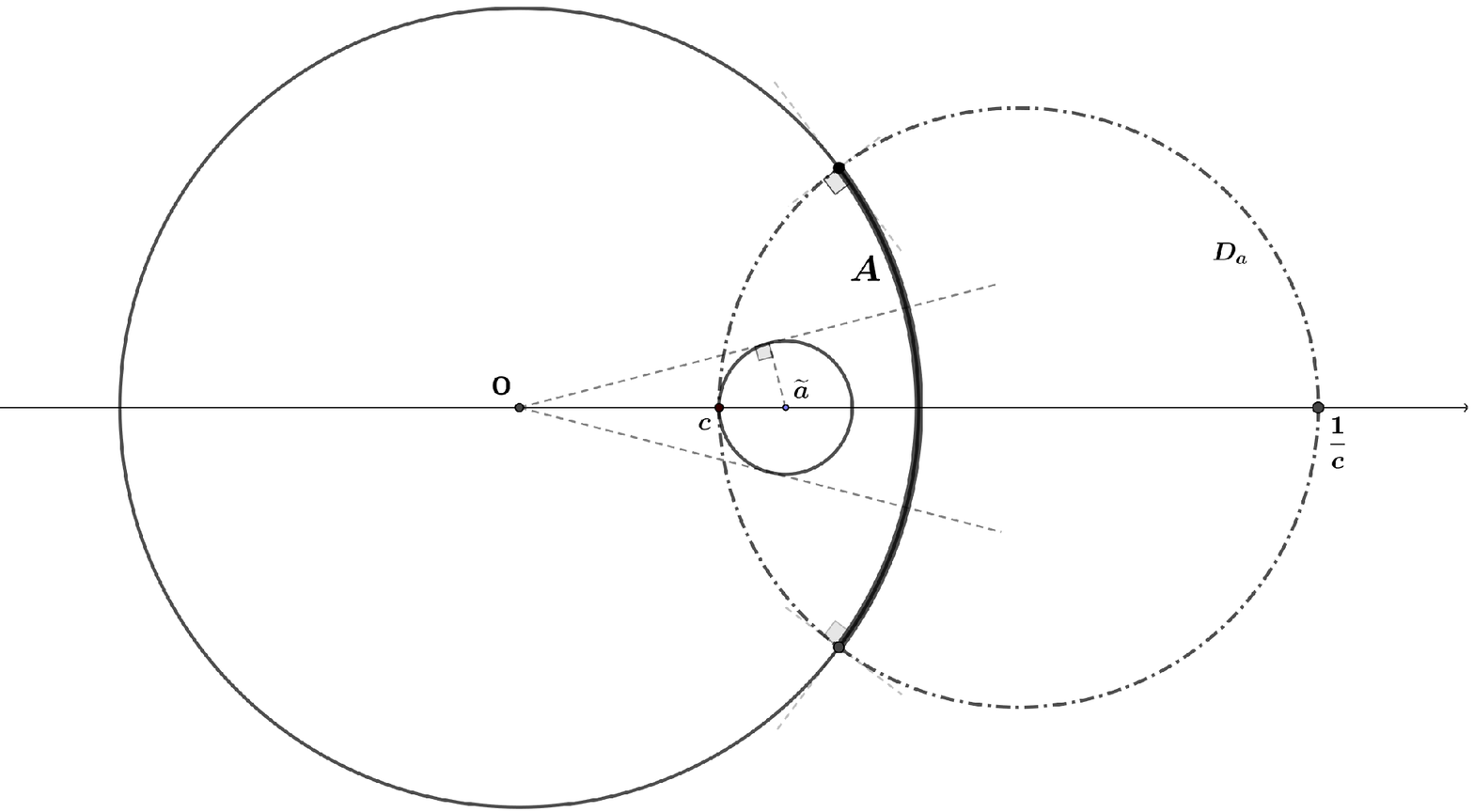}
\caption {\it circles} \label{figure}
\end{figure}
\begin{lemma} \label{lemma Poisson}
Let $0 < \theta_0 < \pi /2$ and $A$ be the arc of $\partial \D$ with end points $\e^{- i \theta_0}$ and $\e^{i \theta_0}$ and midpoint $1$. Then, if 
$D_A$ is the disk orthogonal to $\partial \D$ passing through $\e^{- i \theta_0}$ and $\e^{i \theta_0}$, we have 
\begin{displaymath} 
\qquad {\rm P} [\ind_A] \geq 1 /2 \quad \text{on} \quad \D \cap D_A \, .
\end{displaymath} 
\end{lemma}
\begin{proof}
Let $T \colon \overline{\D} \to \{z \in \C \tq \Im z \geq 0 \} \cup \{\infty\}$ be the conformal map mapping $A$ onto $\R_- \cup \{\infty\}$ and 
$\partial \D \setminus A$ onto $\R_+^\ast$.  The unique bounded solution to the Dirichlet problem with data $1$ on $\R_- \cup \{\infty\}$ and $0$ on 
$\R_+^\ast$ is $U (\zeta) = \frac{1}{\pi} \, \arg \zeta$; hence ${\rm P} [\ind_A] (z) = \frac{1}{\pi} \, \arg (T z)$. 

We have $U (\zeta) \geq 1 / 2$ if and only if $\zeta$ is in the closed left-hand side upper quadrant $Q$. But $T^{- 1} (i \R)$ is the arc orthogonal to $\partial \D$ 
and passing through $\e^{- i \theta_0}$ and $\e^{i \theta_0}$; hence $T^{- 1} (i \R) = A$ and,  since $D_A$ is orthogonal to $\D$, we have 
$T \big( \overline{\D \cap D_A} \big) =  Q \cup \{\infty\}$. Therefore ${\rm P} [\ind_A] (z) \geq 1 / 2$ when $z \in \D \cap D_A$.
\end{proof}

By this lemma, we have
\begin{displaymath} 
\qquad u_a \geq \log 2 \quad \text{on } \D \cap D_a \, .
\end{displaymath} 
In particular, since $\Delta (a, 1 / 2) \subseteq \D \cap D_a$, we have $u_a \geq \log 2$ on $\partial \Delta (a, 1 / 2)$. Of course $u_a$ is equal to $\ind_A$ 
on $\partial \D$, so is positive on $\partial \D$.

On the other hand, the function $\log ( 1 / |\phi_a|)$ is harmonic in $\D \setminus \Delta (a, 1 / 2)$ and is equal to $0$ on $\partial \D$ and to 
$\log 2$ on $\partial \Delta (a, 1 / 2)$. Therefore, since $\partial \D \cup \partial \Delta (a, 1 / 2)$ is the boundary of $\D \setminus \Delta (a, 1 / 2)$, we 
obtain that
\begin{equation} \label{majo 2}
\qquad \quad u_a (z) \geq \log \big( 1 / |\phi_a (z)| \big) \quad \text{for } z \in \D \setminus \Delta (a, 1 / 2) \, .
\end{equation} 

$\bullet$ It follows from \eqref{majo 1} and \eqref{majo 2} that
\begin{equation} \label{majo}
\qquad \log \big( 1 / |\phi_a (z) | \big) \leq f_a (\arg z) + u_a (z) \quad \text{ for all } z \in \D \, .
\end{equation}

$\bullet$ We are now able to finish the proof.
\smallskip

We write 
\begin{displaymath} 
B_0 = \prod_{n = 1}^\infty \frac{|a_n|}{a_n} \, \phi_{a_n} 
\end{displaymath} 
with $\sum_{n = 1}^\infty (1 - |a_n|) < \infty$ and $|a_n | \geq 3 / 4$. We have, by \eqref{majo}:
\begin{displaymath}
\log \big( 1 / |B_0 (z)| \big) \leq \sum_{n = 1}^\infty f_{|a_n|} \big( \arg (\bar{a}_n z ) \big) 
+ \sum_{n = 1}^\infty u_{|a_n|} ( z  \, \e^{- i \arg a_n}) \, ,
\end{displaymath}
that is
\begin{equation} 
\log \big( 1 / |B_0 (z)| \big) \leq f (\arg z) + u (z) \, ,
\end{equation} 
with 
\begin{displaymath}
f (\theta) = \sum_{n = 1}^\infty f_{|a_n|} \big( \arg ( \bar{a}_n \,\e^{i \theta}) \big) 
\end{displaymath}
and $u = {\rm P} [w]$, where
\begin{displaymath}
w = 2 \log 2 \sum_{n = 1}^\infty \ind_{\e^{i \arg a_n} A_{|a_n|}} \, .
\end{displaymath}
We have $f , w \in L^1 (- \pi, \pi)$, by invariance of the Lebesgue measure and since, we have  $\| f_{|a_n|} \|_1 \leq C \, (1 - |a_n|)$ and 
$\| w_{|a_n|} \|_1 \leq C \, (1 - |a_n|)$, and $\sum_{n = 1}^\infty (1 - |a_n|) < \infty$. That finishes the proof of 
Proposition~\ref{proposition Blaschke}. 
\end{proof}

Now, it is easy to end the proof of Theorem~\ref{theo with inf} and hence that of Theorem~\ref{theo HS suff-ness}.

\begin{proof} [End of the proof of Theorem~\ref{theo with inf}] 
We only have to write
\begin{displaymath} 
B = \bigg( \prod_{|a_n| < 3 / 4} \frac{|a_n |}{a_n} \, \phi_{a_n} \bigg) \times B_0 
\end{displaymath} 
where $B_0$ is the Blaschke product made with the zeros of $B$ of modulus $\geq 3 / 4$ (as usual $|a_n| / a_n = 1$ if $a_n  =0$). 

Then, with the notation of Lemma~\ref{essential lemma}, if 
\begin{displaymath} 
G = \sum_{|a_n| < 3 / 4} G_{a_n} \, ,
\end{displaymath} 
we have, by Proposition~\ref{proposition Blaschke}, if $|z| \geq 1/2$:
\begin{displaymath} 
\log \big( 1 / |B (z) |\big) \leq G (\arg z) + f (\arg z) + u (z) \, , 
\end{displaymath} 
with $f \in L^1 (- \pi, \pi)$ and $u = {\rm P} [w]$ with $w \in L^1 (- \pi, \pi)$. 

Since the maximal radial function of $u = {\rm P} [w]$ is smaller than its Hardy-Littlewood maximal function $M_w$ (actually equivalent: see 
\cite[Theorem~11.20 and Exercise 19]{Rudin RCA}), by a well-known theorem of Hardy and Littlewood, we get:
\begin{equation}
\sup_{1 / 2 \leq r < 1} \log \frac{1}{| B (r \, \e^{i \theta}) |} \leq G (\theta) + f (\theta) + M_w (\theta) \, . 
\end{equation}
Now, $G \in L^1 (- \pi, \pi)$, by Lemma~\ref{essential lemma}, and $M_w \in L^{1, \infty} (- \pi, \pi)$, by the Kolmogorov theorem; therefore 
$\log ( 1 / |B|) \in L^{(1, \infty)}  (- \pi, \pi)$,  and that finishes the proof of Theorem~\ref{theo with inf}.
\end{proof}

\noindent{\bf Remark.} The proofs of Theorem~\ref{theo HS suff-ness}, $2)$ and Theorem~\ref{theo with inf} show that if the weighted Bergman space 
${\mathfrak B}_U^{\, 2}$ of analytic functions $f$ such that $\int_\D |f|^2 \, U \, dA < \infty$, contains a function $v \in H^\infty$, with 
$v (0) = 1$, then $\log (1 / I_U) \in L^{1, \infty} (0, 2 \pi)$.
\medskip

\goodbreak

The result of Theorem~\ref{theo with inf} is essentially sharp, as said by the following result.

\begin{theorem} 
There exists $v \in H^\infty$, $v \not\equiv 0$, such that $\log (1 / I_v) \notin L^1 (0, 2 \pi)$. 
\end{theorem} 
\begin{proof}
We start with
\begin{displaymath} 
\sigma (\theta) = \left\{ 
\begin{array}{cl}
1 / [ \theta (\log \theta)^2 ] \, , & 0 < \theta \leq 1 /2 \, , \\
0 & \text{elsewhere.} 
\end{array}
\right.
\end{displaymath} 
We have $\sigma \in L^1 (0, 2 \pi)$, and we consider $u = {\rm P} [\sigma]$. Then $u$ is positive and, since the Poisson kernel is positive and decreasing 
on $[0, \pi]$, we have, for $0 < \theta \leq 1 / 2$:
\begin{align*}
u ( \rho \, \e^{i \theta}) 
& = \frac{1}{2 \pi} \int_0^{2 \pi} {\rm P}_\rho (\theta - t) \, \sigma (t) \, dt 
\geq  \frac{1}{2 \pi} \int_0^\theta {\rm P}_\rho (\theta - t) \, \sigma (t) \, dt \\
& \geq  \frac{1}{2 \pi} \int_0^\theta {\rm P}_\rho (\theta) \, \sigma (t) \, dt 
= \frac{1}{2 \pi} \, {\rm P}_\rho (\theta) \, \frac{1}{\log (1 / \theta)} \, \cdot
\end{align*}
Taking $\rho = 1 - \theta$, we have, as $\theta$ goes to $0$:
\begin{displaymath} 
\frac{1 - \rho^2}{1 - 2 \rho \, \cos \theta + \rho^2} = \frac{1 - \rho^2}{(1 - \rho)^2 + 2 \rho (1 - \cos \theta)} 
\sim \frac{\theta (2 - \theta)}{\theta^2 + 2 (1 - \theta) \, \theta^2 / 2} \sim \frac{1}{\theta} \, \cdot
\end{displaymath} 
Hence:
\begin{displaymath} 
u \big( (1 - \theta) \, \e^{i \theta} \big) \geq \frac{C}{\theta} \, \frac{1}{\log (1 / \theta)} \, \cdot 
\end{displaymath} 
Therefore
\begin{displaymath} 
\sup_{1 / 2 \leq \rho < 1} u (\rho \, \e^{i \theta}) \geq \frac{C}{\theta \log (1 / \theta)} \, \cdot
\end{displaymath} 

Let now $g = u + i \tilde u$ and $v = \exp (- g)$. We have $|v | = \e^{- u}$ and hence 
\begin{displaymath} 
I_v (\theta) \leq \exp \bigg( - \frac{C}{\theta \log (1 / \theta) } \bigg) 
\end{displaymath} 
and 
\begin{displaymath} 
\log (1 / I_v) \geq \frac{C}{\theta \log (1 / \theta) } \, \cdot
\end{displaymath} 
Therefore $\log (1 / I_v)  \notin L^1 (0, 2 \pi)$. 
\end{proof}
%

\subsection{An example}

The fact that, in Theorem~\ref{theo HS suff-ness}, the function $U$ has the particular form given in \eqref{def U}, in particular is logarithmically-subharmonic, 
is important. In fact, we have the following result.

\begin{theorem} \label{theo continuous function}
There exist a continuous function $U \colon \D \to \C$, with $U \geq 1$ and an analytic function $w \colon \D \to \C$, $w \not\equiv 0$, such that:
\begin{displaymath} 
\int_\D |w|^2 \, U \, dA < \infty \, ,
\end{displaymath} 
but
\begin{displaymath} 
\qquad \qquad \int_0^1 U (r \, \e^{i \theta}) \, dr = \infty \, , \quad \text{for almost all } \theta \,. 
\end{displaymath} 
\end{theorem} 

To prove this, we use the following weak form of a result of Kahane and Katznelson \cite{KK}.

\begin{theorem} [Kahane-Katznelson]
Given any positive increasing function $\omega \colon (0, 1) \to (0 , \infty)$ such that $\omega (r) \converge_{r \to 1} \infty$ and any pair of 
measurable functions $g, h \colon [0, 2 \pi] \to \overline \R = [- \infty, + \infty]$, there exists an analytic function $F \colon \D \to \C$ such that
\begin{itemize}  
\setlength\itemsep{-0.1 em}

\item [$1)$] $\max_{|z| = r} |F (z)| = {\rm o}\, \big( \omega ( r ) \big)$ as $r$ goes to $1$;

\item [$2)$] $\lim_{r \to 1} \Re F (r \, \e^{i \theta}) = g (\theta)$ and $\lim_{r \to 1} \Im F (r \, \e^{i \theta}) = h (\theta)$, for almost all 
$\theta \in [0, 2 \pi]$. 
\end{itemize}
\end{theorem}
\begin{proof} [Proof of Theorem~\ref{theo continuous function}]
The Kahane-Katznelson theorem shows that there exists a function $w = \exp (- F)$ belonging to ${\mathfrak B}^{\, 2}$, and even in 
$\bigcap_{\beta > - 1} {\mathfrak B}_\beta^{\, 2}$, if we want, taking, for instance, $\omega (r) = \sqrt {\log \big( 1 / ( 1 - r)\big)}$, 
such that $\lim_{r \to 1} w (r \, \e^{i \theta}) = 0$ for almost every $\theta \in [0, 2 \pi]$. We may also assume that $w (0) = 1$. 
\smallskip

By the Egorov theorem, we get, for all $n \geq 1$, numbers $\rho_n  \in (1 - \frac{1}{n} \, \raise 0,5 pt \hbox{,} \, 1)$ and measurable sets 
$A_n \subseteq \T$ such that

\begin{itemize}  
\setlength\itemsep{-0.1 em}

\item [1)] $|w (r \, \e^{i \theta}) |^2 \leq 2^{- n}$ for $\e^{i \theta} \in A_n$ and $\rho_n \leq r < 1$;

\item [2)] $m (\T \setminus A_n) < 2^{- n}$. 
\end{itemize}
By the regularity of the measure, we can assume that the sets $A_n$ are closed.

Let $\rho'_n$ and $\rho''_n$ such that $\rho_n < \rho'_n < \rho''_n < \frac{1 + \rho_n}{2}$ and $\rho''_n - \rho'_n \geq (1 - \rho_n) / 3$, and 
\begin{displaymath}
E_n = \{ r\, \e^{i \theta} \tq \e^{i \theta} \in A_n \, , \ \rho'_n \leq r \leq \rho''_n \} \, .
\end{displaymath}
By the continuity of $w$, there is an open neighborhood $G_n$ of $E_n$, with closure contained in $\D$, such that 
$\rho_n < |z| < (1 + \rho_n) / 2$  and $| w ( z) |^2 < 2^{- n + 1}$ for $z \in G_n$.
We have
\begin{displaymath}
\int_{G_n} \frac{2}{1 - \rho_n} \, |w (z)|^2 \, d A (z) 
\leq \frac{1}{\pi} \int_0^{2 \pi} \bigg( \int_{\rho_n}^{ \frac{1 + \rho_n}{2}} \frac{2}{1 - \rho_n} \, 2^{- n + 1} \, dr \bigg) \, d\theta 
\leq 2^{ - n + 2} \, .
\end{displaymath}
The Urysohn lemma gives a continuous function $U_n \colon \D \to \big[ 0, \frac{2}{1 - \rho_n} \big]$ such that:
\begin{itemize}  
\setlength\itemsep{-0.1 em}

\item [a)] $U_n (z) = 0$ if $z \notin G_n$;

\item [b)] $U_n (z) = 2 / (1 - \rho_n)$ for $z \in E_n$;

\item [c)] $\dis \int_\D |w|^2 \, U_n \, dA \leq 2^{ - n + 2}$.
\end{itemize}

Then 
\begin{displaymath}
U = 1 + \sum_{n  = 1}^\infty U_n
\end{displaymath}
is continuous on $\D$, because the sum is locally finite, since $U_n = 0$ out of $G_n$. We have, by c), and since $w \in {\mathfrak B}^{\, 2}$:  
\begin{displaymath}
\int_\D |w|^2 \, U \, d A = \int_\D |w|^2 \, dA + \sum_{n = 1}^\infty \int_\D |w|^2 \, U_n \, dA < \infty \, .
\end{displaymath}
Moreover, since $\sum_{n = 1}^\infty m (\T \setminus A_n) < \infty$, for almost all $\theta$, there exists $N (\theta) \geq 1$ such that 
$\e^{i \theta} \in A_n$ for all $n \geq N (\theta)$. Hence, for these $\theta$:
\begin{align*}
\int_0^1 U ( r \, \e^{i \theta}) \, dr 
& \geq \sum_{n = N (\theta)}^\infty \int_{\rho'_n}^{\rho''_n} U_n ( r \, \e^{i \theta}) \, dr 
\geq \sum_{n = N (\theta)}^\infty \int_{\rho'_n}^{\rho''_n} \frac{2}{1 - \rho_n} \, dr \\
& =  \sum_{n = N (\theta)}^\infty (\rho''_n - \rho'_n) \, \frac{2}{1 - \rho_n} \geq \sum_{n = N (\theta)}^\infty \frac{2}{3} = \infty \, .
\end{align*}
That finishes the proof of Theorem~\ref{theo continuous function}.
\end{proof}
%

\goodbreak\bigskip

\noindent{\bf Acknowledgement.}\! L. Rodr{\'\i}guez-Piazza is partially supported by the projects MTM2015-63699-P and PGC2018-094215-B-I00 
(Spanish Ministerio de Ciencia, Innovaci\'on y Universidades, and FEDER funds). 
 
Parts of this paper were made when he visited the Universit\'e d'Artois in Lens and the Universit\'e de Lille in January 2019 and in January 2020. It is his pleasure 
to thank all his colleagues in these universities for their warm welcome.

This work is also partially supported by the grant ANR-17-CE40-0021 of the French National Research Agency ANR (project Front).

\goodbreak

\smallskip\goodbreak

{\footnotesize
Pascal Lef\`evre \\
Univ. Artois, UR~2462, Laboratoire de Math\'ematiques de Lens (LML), F-62\kern 1mm 300 LENS, FRANCE \\
pascal.lefevre@univ-artois.fr
\smallskip

Daniel Li \\ 
Univ. Artois, UR~2462, Laboratoire de Math\'ematiques de Lens (LML), F-62\kern 1mm 300 LENS, FRANCE \\
daniel.li@univ-artois.fr
\smallskip

Herv\'e Queff\'elec \\
Univ. Lille Nord de France, USTL,  
Laboratoire Paul Painlev\'e U.M.R. CNRS 8524, F-59\kern 1mm 655 VILLENEUVE D'ASCQ Cedex, FRANCE \\
Herve.Queffelec@univ-lille.fr
\smallskip
 
Luis Rodr{\'\i}guez-Piazza \\
Universidad de Sevilla, Facultad de Matem\'aticas, Departamento de An\'alisis Matem\'atico \& IMUS,  
Calle Tarfia s/n  
41\kern 1mm 012 SEVILLA, SPAIN \\
piazza@us.es
}


\begin{thebibliography} {99}

\bibitem {Co-MC} C.~C.~Cowen, B.~D.~MacCluer, 
Composition operators on spaces of analytic functions, 
Studies in Advanced Mathematics, CRC Press, Boca Raton FL (1995). 

\bibitem {Cuc-Zhao} {\u Z}.~{\u C}u{\u c}kovi{\' c}, R.~Zhao, 
Weighted composition operators on the Bergman space, 
J. London Math. Soc. (2) 70, no. 2 (2004), 499--511. 

\bibitem {Duren} P.~L.~Duren, 
Theory of $H^p$ spaces,
Dover Publ. Inc., Mineola, New-York (2000).

\bibitem {GKP} E.~A.~Gallardo-Guti\'errez, R.~Kumar, J.~R.~Partington, 
Boundedness, compactness and Schatten-class membership of weighted composition operators, 
Integral Equations Operator Theory 67, no. 4 (2010), 467--479. 

\bibitem {Garnett} J.~B.~Garnett, 
Bounded analytic functions, Revised first edition, 
Graduate Texts in Mathematics 236, Springer, New York (2007).

\bibitem {Garnett-Marshall} J.~B.~Garnett, D.~E.~Marshall,  
Harmonic measure, Reprint of the 2005 original, 
New Mathematical Monographs 2, Cambridge University Press, Cambridge (2008).

\bibitem {Hastings} W.~W.~Hastings, 
A Carleson measure theorem for Bergman spaces, 
Proc. Amer. Math. Soc. 52 (1975), 237--241. 

\bibitem {KK}  J.-P.~Kahane, Y~Katznelson, 
Sur le comportement radial des fonctions analytiques, 
C. R. Acad. Sci. Paris S\'er. A--B 272 (1971), A718--A719.

\bibitem {GDHL} G.~Lechner, D.~Li, H.~Queff\'elec, L.~Rodr{\'\i}guez-Piazza, 
Approximation numbers of weighted composition operators, 
J.~Funct.~Anal. 274, no. 7 (2018), 1928--1958.

\bibitem{LELIQUR} P. Lef\`evre, D. Li, H. Queff\'elec, L. Rodr{\'\i}guez-Piazza, 
Some revisited results about composition operators on Hardy spaces, 
Revista Mat. Iberoamer. 28, no.~1 (2012), 57--76. 

\bibitem{PDHL}  P. Lef\`evre, D. Li, H. Queff\'elec, L. Rodr{\'\i}guez-Piazza, 
Compact composition operators on Bergman-Orlicz spaces, 
Trans. Amer. Math.~Soc. 365, no.~8 (2013), 3943--3970.

\bibitem {LLQR-compactification} P.~Lef\`evre, D.~Li, H.~Queff\'elec, L.~Rodr{\'\i}guez-Piazza, 
Compactification, and beyond, of composition operators on Hardy spaces by weights, 
Ann. Acad. Sci. Fenn. Math. {\it to appear}.

\bibitem {LLQR-comparison} P.~Lef\`evre, D.~Li, H.~Queff\'elec, L.~Rodr{\'\i}guez-Piazza, 
Comparison of singular numbers of composition operators on different Hilbert spaces of analytic functions, 
J. Funct. Anal. 280, no.~3 (2021), article 108834 -- https://doi.org/10.1016/j.jfa.2020.108834.

\bibitem {LD} D.~Li, 
Compact composition operators on Hardy-Orlicz and Bergman-Orlicz spaces, 
Rev. R. Acad. Cienc. Exactas Fís. Nat. Ser. A Mat. RACSAM 105, no. 2 (2011), 247--260.  

\bibitem {LQR} D.~Li, H.~Queff\'elec, L.~Rodr{\'\i}guez-Piazza, 
Approximation numbers of  composition operators, 
J.~Approx.~Theory. 164  (2012), 431--459.

\bibitem {MacCluer-Shapiro} B.~D.~MacCluer, J.~H.~ Shapiro, 
Angular derivatives and compact composition operators on the Hardy and Bergman spaces,
Canad. J. Math. 38, no.~4 (1986), 878--906. 

\bibitem {Moorhouse} J.~Moorhouse, 
Compact differences of composition operators, 
J. Funct. Anal. 219 (2005), no. 1, 70--92. 

\bibitem{Rudin} W.~Rudin, 
Fourier analysis on groups, 
Reprint of the 1962 original, 
Wiley Classics Library, A Wiley-Interscience Publication. John Wiley \& Sons, Inc., New York (1990).

\bibitem{Rudin RCA} W.~Rudin, 
Real and complex analysis, Third edition, 
McGraw-Hill Book Co., New York (1987).

\bibitem {Shapiro-livre} J.~H.~Shapiro, 
Composition operators and classical function theory, 
Universitext, Tracts in Mathematics, Springer-Verlag, New York (1993).


\bibitem {Zyg}  A.~Zygmund, 
Trigonometric series, Vol. I and II, third edition, 
Cambridge Mathematical Library, Cambridge University Press, Cambridge (2002).

\end{thebibliography}
\end{document}